\documentclass{article}
\usepackage{amsmath,amssymb,amsfonts,amsthm,mathrsfs,color,graphics}
\usepackage[all]{xy}
 \usepackage[dvips]{graphicx}
\usepackage[a4paper]{geometry}
\usepackage{textcomp}
\newtheorem{prop}{Proposition}[section]

\newtheorem{theorem}{Theorem}

\theoremstyle{remark}

\theoremstyle{remark}

\theoremstyle{definition}
\newtheorem{definition}[prop]{Definition}

\theoremstyle{plain}
\newtheorem{lemma}[prop]{Lemma}

\newtheorem{corollary}[prop]{Corollary}

\newcounter{exercise}[section]

\newcommand{\mc}{\mathcal}
\newcommand{\mf}{\mathfrak}
\newcommand{\mb}{\mathbb}
\newcommand{\ms}{\mathscr}
\newcommand{\ra}{\rightarrow}

\newcommand{\tx}{\textrm}

\newcommand{\te}[1]{\text{\textnormal{#1}}}

\newtheorem{art}[prop]{}
\newcommand{\Yan}{{Y^{\rm an}}}

\newcommand{\kcirc}{{K^\circ}}
\newcommand{\ktilde}{{\widetilde{K}}}
\newcommand{\R}{{\mathbb R}}
\newcommand{\Z}{{\mathbb Z}}
\newcommand{\C}{{\mathbb C}}
\newcommand{\Acal}{{\mathscr A}}

\newcommand{\Ycal}{{\mathscr Y}}
\newcommand{\Yfrak}{{\mf Y}}

\newcommand{\Zcal}{{\mathscr Z}}
\newcommand{\Ucal}{{\mathscr U}}

\newcommand{\Vcal}{{\mathscr V}}
\newcommand{\Mcal}{{\mathscr M}}
\newcommand{\Lcal}{{\mathscr L}}
\newcommand{\Div}{{\rm div}}
\newcommand{\cyc}{{\rm cyc}}
\newcommand{\Pic}{{\rm Pic}}
\newcommand{\ord}{{\rm ord}}
\newcommand{\Spec}{{\rm Spec}}
\newcommand{\Spf}{{\rm Spf}}

\newcommand{\ve}{{\varepsilon}}
\newcommand{\id}{{\rm id}}

\newcommand{\supp}{{\rm supp}}
\newcommand{\Ocal}{{\mathscr O}}
\newcommand{\codim}{{\rm codim}}
\newcommand{\Ycirc}{{Y^\circ}}
\newcommand{\Dcal}{{\mathscr D}}
\newcommand{\tdop}{{\mathbb T}}
\newcommand{\Tor}{{\mathbb G}_m^n}
\newcommand{\pdop}{{\mathbb P}}

\newcommand{\yb}{{\mathbf y}}

\newcommand{\cdop}{{\mathbb C}}

\newcommand{\qdop}{{\mathbb Q}}

\newcommand{\rdop}{{\mathbb R}}

\newcommand{\zdop}{{\mathbb Z}}

\newcommand{\Tan}{{T^{\rm an}}}

\begin{document}
\title{Classification of normal toric varieties \\ over a valuation ring of rank one}
\author{Walter Gubler and Alejandro Soto}
\maketitle

\begin{abstract}
Normal toric varieties over a field or a discrete valuation ring are classified by rational polyhedral fans. We generalize this classification to normal toric varieties over an arbitrary valuation ring of rank one. The proof  is based on a generalization of Sumihiro's theorem to this non-noetherian setting. These toric varieties play an important role for tropicalizations.

\vspace{2mm}
\end{abstract}
\tableofcontents

\section{Introduction}

Toric varieties over a field have been studied since the 70's. Their geometry is completely determined by the convex geometry of rational polyhedral cones. This gives toric geometry an important role in algebraic geometry for testing conjectures.  There are many good  references for them, for instance Cox--Little--Schenk \cite{cox}, Ewald \cite{ewald}, Fulton \cite{fulton93},  Kempf--Knudsen--Mumford--Saint-Donat \cite{mumford73} and Oda \cite{oda}. Although in these books, toric varieties are defined over an algebraically closed field, the main results  hold over any field.   

Motivated by compactification and degeneration problems, Mumford considered in \cite[Chapter IV]{mumford73} normal toric varieties over discrete valuation rings. A similar motivation was behind Smirnov's paper \cite{smirnov} on projective toric varieties over discrete valuation rings. In the paper of {Burgos--Philippon--Sombra} \cite{bps}, toric varieties over discrete valuation rings were considered for applications to an arithmetic version of the famous Bernstein--Kusnirenko--Khovanskii theorem in toric geometry. The restriction to discrete valuation rings is mainly caused by the use of standard methods from algebraic geometry requiring noetherian schemes. 

At the beginning of the new century, tropical geometry emerged as a new branch of mathematics (see \cite{sturm} or \cite{Mik}). We fix now a valued field $(K,v)$ with value group $\Gamma:=v(K^\times) \subset \rdop$. The Bieri--Groves theorem shows that the tropicalization of a closed $d$-dimensional subvariety of a split torus $T:=(\Tor)_K$ over $K$ is a finite union of $\Gamma$-rational $d$-dimensional polyhedra in $\rdop^n$. Moreover, this tropical variety is the support of a weighted polyhedral complex of pure dimension $d$ such that the canonical tropical weights satisfy a balancing condition in every face of codimension $1$. The study of the tropical weights leads naturally to toric schemes over the valuation ring $\kcirc$ of $K$ (see \cite{gubler12} for details). 

A $\tdop$-toric scheme $\Ycal$ over $\kcirc$ is an integral separated flat scheme  over $\kcirc$ containing $T$ as a dense open subset such that the translation action of $T$ on $T$ extends to an algebraic action of the split torus $\tdop:=(\Tor)_{\kcirc}$ on $\Ycal$. If a $\tdop$-toric scheme  is of finite type over $\kcirc$, then we call it a $\tdop$-toric variety. There is an affine $\tdop$-toric scheme $\Vcal_\sigma$ associated to any  $\Gamma$-admissible cone $\sigma$ in $\rdop^n \times \rdop_+$. This construction is {similar to} the classical theory of toric varieties over a field, where every rational polyhedral cone in $\rdop^n$  containing no lines gives rise to an affine $T$-toric variety. The additional factor $\rdop_+$ takes the valuation $v$ into account and $\Gamma$-admissible cones are cones containing no lines satisfying a certain rationality condition closely related to {the} $\Gamma$-rationality in the Bieri--Groves theorem. If $\Sigma$ is a fan in $\rdop^n \times \rdop_+$ of $\Gamma$-admissible cones, then we call $\Sigma$ a $\Gamma$-admissible fan. By using a gluing process along common subfaces, we get an associated $\tdop$-toric scheme $\Ycal_\Sigma$ over $\kcirc$ with the open affine covering $(\Vcal_\sigma)_{\sigma \in \Sigma}$. We refer to \S \ref{Toric schemes over valuation rings} for precise definitions.  These normal $\tdop$-toric schemes $\Ycal_\Sigma$ are studied in \cite{gubler12}. Many of the properties of toric varieties over fields hold also for $\Ycal_\Sigma$.

Rohrer considered in \cite{rohrer12} toric schemes $X_\Pi$ over an arbitrary base $S$ associated to a rational fan $\Pi$ in $\rdop^n$ containing no lines.  If we restrict to the case $S = \Spec(\kcirc)$, then Rohrer's toric schemes are a special case of the above $\tdop$-toric schemes as we have $X_\Pi = \Ycal_{\Pi \times \rdop_+}$. Note that the cones of $\Pi \times \rdop_+$ are preimages of cones in $\rdop^n$ with respect to the canonical projection $\rdop^n \times \rdop_+ \rightarrow \rdop^n$ and hence the fans $\Pi \times \rdop_+ $ form a very special subset of the set of $\Gamma$-admissible fans in $\rdop^n \times \rdop_+$. As a consequence, Rohrer's toric scheme $X_\Pi$ is always obtained by base change from a corresponding toric scheme over $\Spec(\zdop)$ while this is in general not true for $\Ycal_\Sigma$. The generic fibre of $\Ycal_\Sigma$ is the toric variety over $K$ associated to the fan formed by the recession cones of all $\sigma \in \Sigma$, but the special fibre of $\Ycal_\Sigma$ {need not}  be a toric variety. In fact, the special fibre of $\Ycal_\Sigma$ is a union of toric varieties  corresponding to the vertices of the polyhedral complex $\Sigma \cap (\rdop^n \times \{1\})$. On the other hand, every fibre of the toric scheme $X_\Pi$ is a toric variety associated to the same fan $\Pi$. 

This leads to the natural question if every normal $\tdop$-toric variety $\Ycal$ over the valuation ring $\kcirc$ is isomorphic to  $\Ycal_\Sigma$ for a suitable $\Gamma$-admissible fan $\Sigma$ in $\rdop^n \times \rdop_+$. This classification  is possible in the classical theory of normal toric varieties over a field and also in the case of normal toric varieties over a discrete valuation ring. First, one shows that every affine normal $\tdop$-toric variety  over a field or a discrete valuation ring is of the form $\Vcal_\sigma$ for a rational cone $\sigma$ in $\rdop^n \times \rdop_+$ containing no lines and then one uses Sumihiro's theorem which shows that every point in $\Ycal$ has a $\tdop$-invariant affine open neighbourhood  (see \cite{mumford73}). Sumihiro proved his theorem over a field in \cite{sumihiro}. In \cite[Chapter IV]{mumford73}, the arguments were extended to the case of a discrete valuation ring. The proof of Sumihiro's theorem relies on noetherian techniques from algebraic geometry. 

Now we describe the structure and the results of the present paper. For the generalization of the above classification to  normal $\tdop$-toric varieties $\Ycal$ over an arbitrary valuation ring $\kcirc$ of rank $1$, one needs a theory of divisors on varieties over $\kcirc$. This will be done in \S 2. First, we recall some basic facts about normal varieties over $\kcirc$ due to Knaf \cite{knaf}. Then we define the Weil divisor associated to a Cartier divisor $D$ and more generally a proper intersection product of $D$ with cycles. This is based on the corresponding intersection theory of Cartier divisors on admissible formal schemes over $\kcirc$ given in \cite{gubler98}. In \S 3, we recall the necessary facts for toric varieties over $\kcirc$ which were proved in   \cite{gubler12}. In \S 4, we show the following classification for affine normal toric varieties:

\begin{theorem} \label{cone}
If $v$ is  not a discrete valuation, then the map $\sigma\mapsto \ms{V}_\sigma$ defines a bijection between the set of those $\Gamma$-admissible cones $\sigma$ in $\rdop^n \times \mb{R}_+$ for which the vertices of $\sigma \cap (\rdop^n\times \{1\})$ are contained in  $\Gamma^n \times \{1\}$ 
 and the set of isomorphism classes of normal affine $\mb{T}$-toric varieties over the valuation ring $K^\circ$.
\end{theorem}

Similarly {to} the classical case, the proof uses finitely generated semigroups in the character lattice of $T$ and duality of convex polyhedral cones. The new ingredient here is an approximation argument concluding the proof of Theorem \ref{cone}. The additional condition for the vertices of the $\Gamma$-admissible cone $\sigma$  is equivalent to the property that the affine $\tdop$-toric scheme  $\Vcal_\sigma$ is of finite type over $\kcirc$ meaning that $\Vcal_\sigma$ is a $\tdop$-toric variety over $\kcirc$ (see Proposition \ref{finitepresentation}). If $v$ is a discrete valuation, then $\Vcal_\sigma$ is always a $\tdop$-toric variety over $\kcirc$ and hence the condition on the vertices has to be omitted to get the bijective correspondence in Theorem \ref{cone}.

For the globalization of the classification, the main difficulty is the generalization of Sumihiro's theorem. The proof follows the same steps working in the case of fields or discrete valuation rings (see \cite{mumford73}, proof of Theorem 5 in Chapter I and \S 4.3). In \S 5, we show that for every non-empty affine open subset $\Ucal_0$ of a normal $\tdop$-toric variety $\Ycal$ over the valuation ring $\kcirc$ of rank one, the smallest open $\tdop$-invariant subset $\Ucal$ containing $\Ucal_0$  has an effective  Cartier divisor $D$ with support equal to $\Ucal \setminus \Ucal_0$. This is rather tricky in the non-noetherian situation and it  is precisely here, where we use the results on divisors from \S 2. 

In \S 6, we use the Cartier divisor $D$ constructed in the previous section to show that $\Ocal(D)$ is an ample invertible sheaf with a $\tdop$-linearization. This leads to a $\tdop$-equivariant immersion of $\Ucal$ into a  projective space over $\kcirc$ on which $\tdop$-acts linearly. It remains to prove Sumihiro's theorem for projective $\tdop$-toric subvarieties of a projective space over $\kcirc$ on which $\tdop$-acts linearly. This variant of Sumihiro's theorem will be proved in \S 7 and relies on properties of such non-necessarily normal projective $\tdop$-toric varieties given in \cite[\S 9]{gubler12}. We get the following generalization of {\it Sumihiro's theorem}:

\begin{theorem}\label{sumihiro1}
Let $v$ be a real valued valuation with valuation ring $\kcirc$ and let  $\ms{Y}$ be a normal $\mb{T}$-toric variety over $\kcirc$.  Then every point of $\Ycal$ has an affine open $\mb{T}$-invariant neighborhood.
\end{theorem}

As an immediate consequence, we will obtain our main classification result:

\begin{theorem}\label{sumihiro2}
If $v$ is not a discrete valuation, then the map $\Sigma \mapsto \ms{Y}_{\Sigma}$ defines a bijection between the set of  fans in $\rdop^n \times \rdop_+$, whose cones are as in Theorem \ref{cone}, and the set of isomorphism classes of normal $\mb{T}$-toric varieties over $K^\circ$.
\end{theorem}

If $v$ is a discrete valuation, then we have to omit the additional condition on the vertices of the cones again to get a bijective correspondence in Theorem \ref{sumihiro2}.

{The authors thank the referee for his comments and suggestions.}

\begin{center}
{\it Notation}
\end{center}

For sets,  in $A\subset B$  equality is not excluded and $A\backslash B$ denotes the complement of $B$ in $A$. The set of non-negative numbers in $\zdop$, $\qdop$ or $\rdop$ is denoted by  $\mb{Z}_+$, $\qdop_+$ or $\mb{R}_+$, respectively. All the rings and algebras are commutative with unity. For a ring $A$, the group of units is denoted by $A^\times$. A variety over a field $k$ is an irreducible and reduced scheme which is separated and of finite type over $k$. See \S \ref{divisors on varieties} for the definition of varieties over a valuation ring.


In the whole paper, we fix a {\it valued field} $(K,v)$ which means here that $v$ is a valuation on the field $K$  with value group $\Gamma:=v(K^\times)\subset \mb{R}$. Note that $K$ is not required to be algebraically closed or complete and that its valuation can be trivial. We have a valuation ring $K^\circ:=\{x\in K\mid v(x)\geq 0\}$ with maximal ideal $K^{\circ \circ}:=\{x\in K\mid  v(x)>0\}$ and residue field $\widetilde{K}:=K^\circ / K^{\circ \circ}$. We denote by $\overline{K}$ an algebraic closure of $K$.



Let $M$ be a free abelian group of rank $n$ with dual $N:=\tx{Hom}(M,\mb{Z})$.  For $u\in M$ and $\omega \in N$, the natural pairing is denoted by $\langle u, \omega \rangle :=\omega (u) \in \mb{Z}$. For  an abelian group $G$, the base change is denoted by $M_G:=M\otimes_{\mb{Z}} G$; for instance $M_{\mb{R}}=M\otimes_{\mb{Z}} \mb{R}$. The split torus over $K^\circ$ of rank $n$ with generic fiber $T=\tx{Spec}(K[M])$ is given by $\mb{T}=\tx{Spec}(K^\circ[M])$, therefore $M$ can be seen as the character lattice of $T$ and $N$ as the group of one parameter subgroups. For $u\in M$, the corresponding character is denoted by $\chi^u$. 


\vspace{2mm} 

\section{Divisors on varieties over the valuation ring} \label{divisors on varieties}

 The goal of this section is to recall some facts about divisors on 
varieties over the valuation ring $\kcirc$ of the valued field $(K,v)$ with value group $\Gamma \subset \rdop$. The problem here is that $\kcirc$ {need not} be noetherian and so we cannot use the usual constructions from algebraic geometry. Instead we will adapt the intersection theory with Cartier divisors on admissible formal schemes from \cite{gubler98} to our algebraic framework. 

{We will start with some topological considerations of varieties over $\kcirc$. As we mainly focus on normal varieties in this paper, we will gather some results of Knaf relating the local rings of such varieties to valuation rings. This will be useful to define the Weil divisor associated to a Cartier divisor. If the variety is not normal then we will pass to the formal completion along the special fibre to define the associated Weil divisor according to \cite{gubler98}. This construction will be explained in \ref{associated Weil divisor}--\ref{multiplicity of D in V} and is rather technical. In Proposition \ref{order for normal}, we will show that both constructions agree for normal varieties over $\kcirc$. As a consequence, we will obtain a proper intersection theory of Cartier divisors with cycles on a (normal) variety over $\kcirc$. The relevant properties are listed in \ref{projection formula}--\ref{equivalence}. }

This intersection theory will be used in the proof of the generalization of Sumihiro's theorem given in \S \ref{Construction of the Cartier divisor}--\ref{Proof of Sumihiro}. {There we will construct a torus-invariant ample divisor on any normal toric variety over $\kcirc$ which will be helpful to reduce Sumihiro's theorem to an easier projective variant.}

\begin{art} \rm \label{varieties}
A {\it variety} over $\kcirc$ is an integral scheme which is of finite type and separated over $\kcirc$. By \cite[Lemma 4.2]{gubler12}  such a variety $\Ycal$   is flat over $\kcirc$. We have $\Spec(\kcirc)=\{\eta,s\}$, where the {\it generic point $\eta$} (resp. the {\it special point $s$}) is the zero-ideal (resp. the maximal ideal) in $\kcirc$. We get the generic fibre $\Ycal_\eta$ as a variety over $K$ and the special fibre $\Ycal_s$ as a separated scheme of finite type over $\ktilde$. The variety $\Ycal$ is called {\it normal} if all the local rings $\Ocal_{\Ycal,y}$ are integrally closed.
\end{art}

\begin{prop} \label{noetherian topological space}
A variety $\Ycal$ over $\kcirc$ is a noetherian topological space. If $d:=\dim(\Ycal_\eta)$, then  every irreducible component of the special fibre has also dimension $d$. If $\Ycal_s$ is non-empty and if $v$ is non-trivial, then the topological dimension of $\Ycal$ is $d+1$. If $\Ycal_s$ is empty or if $v$ is trivial, then $\Ycal = \Ycal_\eta$. 
\end{prop}

\proof The set $\Ycal$ over $\kcirc$ is the  union of $\Ycal_\eta$ and $\Ycal_s$. This proves the first claim. The second claim follows from flatness of $\Ycal$ over $\kcirc$. The other claims are now obvious. \qed
 
\vspace{2mm}

The following facts about normal varieties over a valuation ring follow from a paper by Knaf \cite{knaf}.

\begin{prop}\label{intersection of local rings}
Let $\Ycal$ be a normal  variety over $K^\circ$. Then the following properties hold: 
\begin{itemize}
\item[(a)] For  $y \in \Ycal$, the local ring $\Ocal_{\Ycal,y}$ is a valuation ring if and only if $y$ is either a dense point of a divisor of the generic fibre $\Ycal_\eta$ or $y$ is a generic point of $\Ycal_s$ or $\Ycal_\eta$.
\item[(b)] If $y$ is a dense point of a divisor of the generic fibre, then $\Ocal_{\Ycal,y}$ is a discrete valuation ring.
\item[(c)] If $y$ is a generic point of the special fibre $\Ycal_s$, then $\Ocal_{\Ycal,y}$ is the valuation ring of a real-valued valuation $v_y$ {extending} $v$ such that $\Gamma$ is of finite index in the value group of $v_y$. 
\item[(d)] If $\Ycal=\Spec(A)$, then $A = \bigcap_y \Ocal_{\Ycal,y}$, where $y$ ranges over all points from (b) and (c).
\end{itemize} 
\end{prop}

\proof Since $\Ycal_\eta$ is a normal variety over the field $K$, it is regular in codimension $1$ and hence (b) follows. The claims (a) and (c) follow from \cite[Theorem 2.6]{knaf}, {where the set of valued points of a normal variety over an arbitrary valuation ring are classified}.  It remains to prove (d).  By \cite[Theorem 2.4]{knaf}, the integral domain $A$ is integrally closed and coherent. It follows from \cite[1.3]{knaf} that $A$ is a Pr\"ufer $v$-multiplication ring and hence (d) is a consequence of \cite[1.5]{knaf}. \qed

\begin{art} \rm \label{cycles}
Let $\Ycal$ be a variety over $\kcirc$ with generic fibre $Y$. A {\it horizontal cycle} $Z$ on $\Ycal$ is a cycle on $Y$, i.e. $Z$ is a $\Z$-linear combination of closed subvarieties $W$ of $Y$. The {\it support}  $\supp(Z)$ is the union of all closures $\overline{W}$ in $\Ycal$, where $W$ ranges over all closed subvarieties with non-zero coefficients. Such $W$'s are called {\it prime components} of the horizontal cycle $Z$. 
If the closure of every prime component of $Z$ in $\Ycal$ has dimension $k$ (resp. codimension $p$), then we say that the horizontal cycle $Z$ of $\Ycal$ has {\it dimension} $k$ (resp. {\it codimension $p$}). 

A {\it vertical cycle} $V$ on $\Ycal$ is a cycle on $\Ycal_s$ with real coefficients, i.e. $V$ is an $\R$-linear combination of closed subvarieties $W$ of $\Ycal_s$. The support and prime components are defined as usual. We say that the vertical cycle $V$ of $\Ycal$ has {\it dimension $k$} (resp. codimension $p$) if every prime component of $V$ is a closed subvariety of $\Ycal_s$ of dimension $k$ (resp. of codimension $p$ in $\Ycal$). 

A {\it cycle} $\Zcal$ on $\Ycal$ is a formal sum of a horizontal cycle $Z$ and a vertical cycle $V$. The {\it support} of $\Zcal$ is $\supp(\Zcal):=\supp(Z) \cup \supp(V)$. If the horizontal part $Z$ and the vertical part $V$ of $\Zcal$ both have dimension $k$ (resp. codimension $p$), then we say that $\Zcal$ has {\it dimension $k$} (resp. {\it codimension $p$}).  We say that a cycle is {\it effective} if the multiplicities in all its prime components are positive. 
\end{art}

\begin{art} \rm \label{flat pullback}
If $\varphi: \Ycal' \rightarrow \Ycal$ is a flat morphism of varieties over $\kcirc$, then we define the pull-back $\varphi^*(\Zcal)$ of a cycle $\Zcal$ on $\Ycal$ by using flat pull-back of the horizontal and vertical parts. The resulting cycle $\varphi^*(\Zcal)$ of $\Ycal'$ keeps the same codimension as $\Zcal$. 
Similarly, we define the {\it push-forward} of a cycle with respect to a proper morphism of varieties over $\kcirc$. This preserves the dimension of the cycles.
\end{art}

\begin{art} \rm \label{Cartier divisors}
We recall that the {\it support $\supp(D)$} of a Cartier divisor $D$ on $\Ycal$ is the complement {of the union of all open subsets $\Ucal$ for which $D$ is given by an invertible element in $\Ocal(\Ucal)$.}
Clearly, $\supp(D)$ is a closed subset of $\Ycal$. We say that {\it  the Cartier divisor $D$  intersects {a} cycle $\Zcal$ of $\Ycal$ properly}, if for every prime component $W$ of $\Zcal$, we have $$\codim(\supp(D) \cap \overline{W}, \Ycal) \geq \codim(\overline{W},\Ycal) + 1.$$
\end{art}

\begin{art} \rm \label{associated Weil divisor}
We are going to define the associated Weil divisor $\cyc(D)$ of a Cartier divisor $D$ on the variety  $\Ycal$ over $\kcirc$.  The horizontal part of $\cyc(D)$ is the Weil divisor corresponding to the Cartier divisor $D|_Y$ on the generic fibre $Y$ of $\Ycal$. Thus we just need to construct the vertical part of $\cyc(D)$. This will be done by using the corresponding construction for admissible formal schemes over the completion of $\kcirc$ given in \cite{gubler98}. This is technically rather demanding and we will freely use the terminology and  the results from  \cite{gubler98}. The reader might skip the details below in a first read trusting that the algebraic intersection theory with Cartier divisors works in the usual way. In fact, we are dealing mostly with normal varieties over $\kcirc$ in this paper and then one can use Proposition \ref{order for normal} to define the multiplicities $\ord(D,V)$ of  $\cyc(D)$ in an irreducible component $V$ of $\Ycal_s$ without bothering about admissible formal schemes.

To define the vertical part of $\cyc(D)$, we may assume that $v$ is non-trivial.  Since the special fibre remains the same by base change to the completion of $\kcirc$, we may also assume that $v$ is  complete. Let $\hat{\Ycal}$ be the formal completion of $\Ycal$ along the special fibre (see \cite[\S 6]{ulrich95}). This is an admissible formal scheme over $\kcirc$ with special fibre equal to $\Ycal_s$. We will denote its generic fibre by $\Ycirc$ which is an analytic subdomain of the analytification $\Yan$ of $Y$. Note that $\Ycirc$ may be seen as the set of potentially integral points (see \cite[4.9--4.13]{gubler12} for more details). We have a morphism $\hat{\Ycal} \rightarrow \Ycal$ of locally ringed spaces and using pull-back, we see that the Cartier divisor $D$ induces a Cartier divisor $\hat{D}$ on $\hat{\Ycal}$. By \cite[Definition 3.10]{gubler98},  we have a Weil divisor $\cyc(\hat{D})$ on $\hat{\Ycal}$. It follows from \cite[Proposition 6.2]{gubler98} that the analytification of the Weil divisor $\cyc(D|_Y)$ restricts to the horizontal part of $\cyc(\hat{D})$. We define the vertical part of $\cyc(D)$ as the vertical part of $\cyc(\hat{D})$. 

The multiplicity of $\cyc(D)$ in an irreducible component $V$ of $\Ycal_s$ is denoted by $\ord(D,V)$. Since $\ord(D,V)$ is linear in $D$,  the map $D \mapsto \cyc(D)$ is a homomorphism from the group of Cartier divisors  to the group of cycles of codimension $1$. It follows from the definitions that $\cyc(D)$ is an effective cycle if $D$ is an effective Cartier divisor. 
For convenience of the reader, we recall the definition of $\ord(D,V)$ in more details to make these statements obvious.
\end{art}

\begin{art} \rm \label{multiplicity of D in V for algebraically closed}
First, we assume that $K$ is algebraically closed and that  $v$ is complete. We repeat that we (may) assume $v$ non-trivial. To define $\ord(D,V)$, we may restrict our attention to an affine neighbourhood of the generic point $\zeta_V$ of $V$ where $D$ is given by a single rational function $f$. Hence we may assume $\Ycal = \Spec(A)$ and $D=\Div(f)$. Then $\hat{\Ycal}$ is the formal affine spectrum of the $\nu$-adic completion $\hat{A}$ of $A$ for any non-zero element $\nu$ in the maximal ideal of $\kcirc$. We note that $\Acal := \hat{A} \otimes_\kcirc K$ is a $K$-affinoid algebra and we have that $Y^\circ$ is the Berkovich spectrum $\Mcal(\Acal)$ of $\Acal$. Since $Y^\circ$ is an analytic subdomain of $Y^{\rm an}$, we conclude that $Y^\circ$ is reduced  (see \cite[Proposition 3.4.3]{berk} {for a proof}). Let $\Acal^\circ$ be the $\kcirc$-subalgebra of power bounded elements in $\Acal$. Then $\Yfrak'':=\Spf(\Acal^\circ)$ is an admissible formal affine scheme over $\kcirc$ with reduced special fibre and we have a canonical morphism $\Yfrak'' \rightarrow \hat{\Ycal}$. The restriction of this morphism to the special fibres  is finite and surjective (see \cite[4.13]{gubler12} for the argument). In particular, there is a generic point $y''$ of $\Yfrak_s''$ over $\zeta_V$. It follows from \cite[Proposition 2.4.4]{berk} that there is a unique $\xi''$ in the generic fibre $Y^\circ$ of $\Yfrak''$ with reduction $y''$. Recall that the elements of $Y^\circ$ are bounded multiplicative seminorms on $\Acal$. Since $y''$ is a   generic point of the special fibre of $\Yfrak''=\Spf(\Acal^\circ)$, the seminorm corresponding to $\xi''$ is in fact an absolute value with valuation ring equal to $\Ocal_{\Yfrak'',y''}$. This follows from \cite[Proposition 6.2.3/5]{bgr}, {where necessary and sufficient conditions are given for the supremum seminorm to be a valuation}. We use it to define 
\begin{equation} \label{ord}
\ord(D,V):= -\sum_{y''} [\ktilde(\overline{y''}):\ktilde(V)]\log|f(\xi'')|,
\end{equation}
where $y''$ is ranging over all generic points of $\Yfrak_s''$ mapping to the generic point $\zeta_V$ of $V$. 
\end{art}

\begin{art} \rm \label{multiplicity of D in V}
If $K$ is not algebraically closed, then we perform base change to $\C_K$. The latter is  the completion of the algebraic closure of the completion of $K$. This is the smallest algebraically closed complete field extending the valued field $(K,v)$, and the residue field $\widetilde{\cdop}_K$ is the algebraic closure of $\ktilde$ \cite[\S  3.4.1]{bgr}.   Again, we may assume $\Ycal = \Spec(A)$ and  $D = \Div(f)$ for a rational function $f$ on $\Ycal$. Let $\Ycal'$ be the base change of ${\Ycal}$ to the valuation ring $\cdop_K^\circ$ of $\C_K$. Let $(\Ycal_j')_{j=1, \dots, r}$ be the irreducible components of $\Ycal'$. Our goal is to define $\ord(D,V)$ in the irreducible component $V$ of $\Ycal_s$. The definition will be determined by the two guidelines that $\cyc(D)$ should be invariant under base change to $\cdop_K^\circ$ and that this base change should be linear in the irreducible components $\Ycal_j'$. 
Since we do not assume  that a variety is geometrically reduced,  the multiplicity $m(Y_j',Y')$ of the generic fibre $Y_j'$ of $\Ycal_j'$ in the generic fibre $Y'$ of $\Ycal'$ has to be considered. Note also that the absolute Galois group ${\rm Gal}(\widetilde{\cdop}_K/\ktilde)$ acts transitively on the irreducible components of the base change $V_{\widetilde{\cdop}_K}$ and hence the multiplicity $m(V',V_{\widetilde{\cdop}_K})$ is independent of the choice of an irreducible component $V'$ of  $V_{\widetilde{\cdop}_K}$.

We choose an irreducible component $V'$ of $V_{\widetilde{\cdop}_K}$. It is also an irreducible component of $\Ycal_s'$ and hence there is an  irreducible component $\Ycal_j'$ containing the generic point $\zeta_{V'}$ of $V'$. For $\Ycal_j'=\Spec(A_j')$, we proceed as in \ref{multiplicity of D in V for algebraically closed}. We get an admissible formal scheme $\Yfrak_j''=\Spf((\Acal_j'')^\circ)$ over $\cdop_K^\circ$ with reduced generic fibre $(Y_j')^\circ = \Mcal(\Acal_j'')$ and reduced special fibre $(\Yfrak_j'')_s$ with a surjective finite map onto $(\Ycal_j')_s$. Hence there is at least one generic point $y_j''$ of $(\Yfrak_j'')_s$ mapping to $\zeta_{V'}$. Again, there is a unique point $\xi_j'' \in  (Y_j')^\circ$ with reduction $y_j''$. Then $\xi_j''$ extends to an absolute value with valuation ring $\Ocal_{\Yfrak_j'',y_j''}$ and \eqref{ord} leads to the definition
\begin{equation} \label{ord2}
\ord(D,V):= -\frac{1}{m(V',V_{\widetilde{\cdop}_K})}    \sum_j m(Y_j',Y') \sum_{y_j''} [{\widetilde{\cdop}_K}(\overline{y_j''}):{\widetilde{\cdop}_K}(V')]\log|f(\xi_j'')|,
\end{equation}
where $Y_j'$ ranges over the irreducible components of $Y'$ and $y_j''$ ranges over the generic points of $(\Yfrak_j'')_s$ lying over the generic point $\zeta_{V'}$ of $V'$. Using the action of ${\rm Gal}(\widetilde{\cdop}_K/\ktilde)$, we see that the definition is independent of the choice of the irreducible component $V'$ of $V_{\widetilde{\cdop}_K}$. It follows from compatibility of base change and passing to the formal completion along the special fibre that $\sum_V \ord(D,V) V$ is indeed the vertical part of $\cyc(D)$ as defined in \ref{associated Weil divisor}.
\end{art}

The Weil divisor associated to a Cartier divisor has all the expected properties. The proofs follow  from the corresponding properties  in \cite{gubler98} or \cite{gubler03}. This is illustrated in the proof of the following {\it projection formula}:

\begin{prop} \label{projection formula}
Let $\varphi:\Ycal' \rightarrow \Ycal$ be a proper morphism of varieties over $\kcirc$ and let $D$ be a Cartier divisor on $\Ycal$ such that $\supp(D)$ does not contain $\varphi(\Ycal')$. As usual, we define $[\Ycal':\Ycal]$ to be the degree of the extension of the fields of rational functions if this degree is finite and $0$ otherwise. Then we have $$\varphi_*(\cyc(\varphi^*(D)))=[\Ycal':\Ycal]\cyc(D).$$
\end{prop}

\proof The projection formula holds in the generic fibre \cite[Proposition 2.3]{fulton98}. We have an induced proper morphism $\hat{\varphi}:\hat{\Ycal}' \rightarrow \hat{\Ycal}$ of admissible formal schemes over the completion of $\kcirc$ and hence the projection formula follows for vertical parts from {the projection formula for proper morphisms of admissible formal schemes given in \cite[Proposition 4.5]{gubler98} and from the compatibility of push-forward and passing to the analytification given in  \cite[Proposition 6.3]{gubler98}}. \qed

\begin{prop} \label{order for normal}
Assume that $v$ is non-trivial, let $\Ycal$ be a normal variety over $\kcirc$  and let $V$ be an irreducible component of $\Ycal_s$. If the Cartier divisor $D$ on $\Ycal$ is given by the rational function $f$ in a neighbourhood of the generic point $y=\zeta_V$ of $V$, then $\Ocal_{\Ycal,y}$ is a valuation ring for a unique real-valued valuation $v_y$ extending $v$ and we have $\ord(D,V)=  v_y(f)$.
\end{prop}

\begin{proof} We may assume that $\Ycal = \Spec(A)$ and $D=\Div(f)$ for a rational function $f$ on $\Ycal$. The first claim follows from Proposition \ref{intersection of local rings}. In the following, we use the notation and the results from \ref{multiplicity of D in V}. We have a generic point $y_j''$ of the special fibre of the admissible formal scheme $\Yfrak_j'':=\Spf((\Acal_j'')^\circ)$ lying over $y$. We have seen  in \ref{multiplicity of D in V} that the unique point $\xi''$ of the generic fibre of $\Yfrak_j''$ mapping to $y_j''$ extends to an absolute value with valuation ring $\Ocal_{\Yfrak_j'',y_j''}$. We conclude that the valuation ring $\Ocal_{\Yfrak_j'',y_j''}$ dominates the valuation ring $\Ocal_{\Ycal,y}$. Since valuation rings are maximal with respect to dominance of local rings in a given field, we conclude that   $-\log|f(\xi'')|= v_y(f)$ and hence \eqref{ord2} simplifies to 
\begin{equation} \label{simplified multiplicity of Cartier divisor in y}
 \ord(D,V) := \frac{v_y(f)}{m(V',V_{\widetilde{\cdop}_K})}    \sum_j m(Y_j',Y') \sum_{y_j''} [{\widetilde{\cdop}_K}(\overline{y_j''}):{\widetilde{\cdop}_K}(V')],
\end{equation}
where $Y_j'$ ranges over the irreducible components of $Y'$ and $y_j''$ ranges over the generic points of $(\Yfrak_j'')_s$ lying over the generic point $\zeta_{V'}$ of $V'$. 
{By \cite[Lemma 4.5]{gubler03}, we have
$$m(V',V_{\widetilde{\cdop}_K})= \sum_j m(Y_j',Y') \sum_{y_j''} [{\widetilde{\cdop}_K}(\overline{y_j''}):{\widetilde{\cdop}_K}(V')]$$}
proving the claim. 
\end{proof}

\begin{corollary} \label{equality of supports}
The following properties hold for a  Cartier divisor $D$ on a  normal variety $\Ycal$ over $\kcirc$. 
\begin{itemize}
\item[(a)] $\supp(D)=\supp(\cyc(D))$. 
\item[(b)] The Cartier divisor $D$ is effective if and only if $\cyc(D)$ is an effective cycle.
\item[(c)] The map $D \mapsto \cyc(D)$ is an injective homomorphism from the group of Cartier divisors on $\Ycal$ to the group of cycles of codimension $1$ on $\Ycal$.
\end{itemize}
\end{corollary}

\begin{proof} It follows easily from the definitions that $\supp(\cyc(D)) \subset \supp(D)$ and that the Weil divisor associated to an effective Cartier divisor is an effective cycle without assuming normality. If $v$ is trivial, the claims are classical results for divisors on normal varieties over $K$ and so we may assume that $v$ is non-trivial. Then (b) follows from Propositions \ref{order for normal} and \ref{intersection of local rings}. To prove (a), the above shows that by passing to the open subset $\Ycal \setminus \supp(\cyc(D)) $, we may assume that $\cyc(D)=0$ and hence (a) follows from (b). Similarly, (c) is a consequence of (b). 
\end{proof}

\begin{art} \label{proper intersection with cycles} \rm
The construction of the Weil divisor associated to a Cartier divisor allows us to define a proper intersection product of a Cartier divisor with a cycle. Indeed, let $D$ be a Cartier divisor intersecting the cycle $\Zcal$ on $\Ycal$ properly. Then we define the {\it proper intersection product $D.\Zcal$} as a cycle on $\Ycal$ in the following way: By linearity, we may assume that $\Zcal$ is a prime cycle $W$. If $W$ is vertical, then $D$ restricts to a Cartier divisor on $W$ and we define $D.W:=\cyc(D|_W)$ using algebraic intersection theory on the variety $W$. If $W$ is horizontal, then $D$ restricts to a Cartier divisor on the closure of $W$ in $\Ycal$ and we define $D.W$ as the associated Weil divisor. Obviously, this proper intersection product is bilinear.  
\end{art}

\begin{prop} \label{commutativity}
Let $D$ and $E$ be properly intersecting Cartier divisors on $\Ycal$  which means  $\codim(\supp(D) \cap \supp(E), \Ycal) \geq 2$. Then we have $D.\cyc(E)=E.\cyc(D)$.  
\end{prop}

\proof For the horizontal parts, this follows from algebraic intersection theory. For the vertical parts, this follows from {the corresponding statement for Cartier divisors on admissible formal schemes given in \cite[Theorem 5.9]{gubler98}. Recall that the special fiber doesn't change after passing to the formal completion of $\Ycal$ along $\Ycal_s$}. \qed

\begin{prop} \label{flat pullback and associated Weil divisor}
Let $\varphi:\Ycal' \rightarrow \Ycal$ be a flat morphism of varieties over $\kcirc$ and let $D$ be a Cartier divisor on $\Ycal$. Then we have $\varphi^*(\cyc(D))=\cyc(\varphi^*(D))$.
\end{prop}

\proof Since $\varphi$ is flat, the pull-back of $D$ is well-defined as a Cartier divisor and the claim follows from \cite[Proposition 4.4(d)]{gubler03}, {where it is proved for Cartier divisors on admissible formal  schemes}. \qed

\begin{art} \rm \label{equivalence}
We say that two cycles $\Dcal_1$ and $\Dcal_2$ of codimension $1$ on the variety $\Ycal$ over $\kcirc$ are {\it rationally equivalent} if there is a non-zero rational function $f$ on $\Ycal$ such that $\Dcal_1 - \Dcal_2 = \cyc(\Div(f))$. The {\it first Chow group} $CH^1(\Ycal)$ of $\Ycal$ is defined as the group of cycles of codimension $1$ modulo rational equivalence. It follows from Proposition \ref{flat pullback and associated Weil divisor} that rational equivalence is compatible with flat pull-back. 

Two Cartier divisors $D_1$ and $D_2$ on $\Ycal$ are said to be {\it linearly equivalent} if there is a non-zero rational function $f$ on $\Ycal$ such that $D_1 - D_2 = \Div(f)$. The group of Cartier divisors modulo linear equivalence is isomorphic to $\Pic(\Ycal)$ using the map $D \mapsto \Ocal(D)$. 

We may use rational equivalence to define a refined intersection theory with pseudo divisors on a variety $\Ycal$ over $\kcirc$ with the same properties as in \cite[Chapter 2]{fulton98}. The proofs follow directly from \cite{gubler98} and \cite[ \S 4]{gubler03}. This will not be used in the sequel and so we leave the details to the interested reader.
\end{art}

\section{Toric schemes over valuation rings} \label{Toric schemes over valuation rings}

In this section, $(K,v)$ is a valued field with valuation ring $\kcirc$ , residue field $\ktilde$ and value group $\Gamma:=v(K^\times)\subset \rdop$. As usual, $\tdop = \Spec(\kcirc[M])$ is the split torus of rank $n$ with generic fibre $T$ and $N$ is the dual of the free abelian group $M$. 

{We will start with the definition of a $\tdop$-toric scheme over $\kcirc$. Then we will recall from \cite{gubler12} the basic construction of a normal $\tdop$-toric scheme associated to a $\Gamma$-admissible fan. The overall goal of this paper is to show that every normal toric variety over $\kcirc$ arises in this way. At the end of this section, we will study projective toric varieties over $\kcirc$ with a linear action of the torus. They can also be understood purely in combinatorial terms, but they are not necessarily normal. We will use them in \S \ref{Proof of Sumihiro} to finish the proof of Sumihiro's theorem.}

\begin{definition}
A {\it $\mb{T}$-toric scheme} over the valuation ring $K^\circ$ is an integral separated flat scheme $\ms{Y}$ over $K^\circ$ such that the generic fiber  $\ms{Y}_\eta$ contains $T$ as an open subset  and such that the translation action $T\times_K T\ra T$ extends to an algebraic action $\mb{T}\times_{K^\circ} \ms{Y} \ra \ms{Y}$ over $K^\circ$. 

A {\it homomorphism} (resp. {\it isomorphism}) of $\tdop$-toric schemes is an equivariant morphism (resp. isomorphism) which restricts to the identity on $T$.
A $\mb{T}$-toric scheme  of finite type over $\kcirc$ is called a {\it $\mb{T}$-toric variety}.  
\end{definition}
Note that if $\ms{Y}$ is a $\mb{T}$-toric variety over $\kcirc$, then $\ms{Y}_\eta$ is a $T$-toric variety over $K$.  
In order to construct examples of $\mb{T}$-toric schemes and to see how they can be described by the combinatorics of some objects in convex geometry, we need to introduce and to study  the following algebras associated to $\Gamma$-admissible cones.

\begin{art} \label{admissible cones} \rm 
A cone $\sigma \subset N_{\mb{R}}\times \mb{R}_+$ is called {\it $\Gamma$-admissible} if it can be written as 
\[ \sigma =\bigcap_{i=1}^k \left\{ (\omega ,s)\in   N_{\mb{R}}\times \mb{R}_+\mid \langle u_i, \omega \rangle +sc_i\geq0 \right\},\quad u_1,\ldots , u_k\in M, c_1,\ldots ,c_k\in \Gamma,\]
and does not contain a line. For such a cone $\sigma$, we define $$K[M]^\sigma:= \{\sum_{u \in  M} \alpha_u \chi^u \in K[M] \mid  c v(\alpha_u) + \langle u, \omega \rangle \geq 0 \; \forall (\omega,c) \in \sigma\}$$
and $\Vcal_\sigma := \Spec( K[M]^\sigma)$. It is easy to see that $K[M]^\sigma$ is an $M$-graded $\kcirc$-algebra and hence we have a canonical $\tdop$-action on $\Vcal_\sigma$. 
\end{art}

\begin{prop}\label{finitepresentation}
Let $\sigma$ be a $\Gamma$-admissible cone  in $N_\R \times \mb{R}_+$. Then $\Vcal_\sigma$ is a normal $\tdop$-toric scheme over $\kcirc$. If $v$ is a discrete valuation, then $\Vcal_\sigma$ is always a $\tdop$-toric variety. If $v$ is not a discrete valuation, then $\Vcal_\sigma$ is a $\tdop$-toric variety over $\kcirc$ if and only if  the vertices of $\sigma \cap (N_\R \times \{1\})$ are contained in $N_\Gamma \times \{1\}$.
\end{prop}
\begin{proof}
{Normality is proven in \cite[Proposition 6.10]{gubler12}. If the valuation is discrete, then  \cite[Proposition 6.7]{gubler12} shows that $\Vcal_\sigma$ is a $\tdop$-toric variety. If $v$ is not discrete, then the equivalence in the last claim is proved in \cite[Propositions 6.9]{gubler12}.}  
\end{proof}

\begin{art} \rm \label{admissible fans}
 A {\it $\Gamma$-admissible fan} $\Sigma$ in $N_{\mb{R}}\times \mb{R}_+$ is a fan consisting of $\Gamma$-admissible cones.
Given a $\Gamma$-admissible fan $\Sigma$, we glue  the normal affine $\mb{T}$-toric schemes $\ms{V}_\sigma$, $\sigma \in \Sigma$, along the open subschemes coming from their common faces. The result is a normal $\tdop$-toric scheme $\Ycal_\Sigma$. Similarly {to} the classical case of toric varieties over a field, the properties of the $\mb{T}$-toric schemes $\Ycal_\Sigma$ may be described by  the combinatorics of the cones $\Sigma$. For details, we refer to \cite{gubler12}.
\end{art}

Now we review the construction of  projective  $\mb{T}$-toric schemes which are not necessarily normal (see \cite[\textsection 9 ]{gubler12} for more details). These are not all the possible projective toric schemes over $K^\circ$ but just those which have a linear action of the torus, see \cite[Proposition 9.8]{gubler12}. For the corresponding projective toric varieties over a field, we refer to  Cox--Little--Schenk \cite[\S 2.1, \S 3.A]{cox} and  Gelfand--Kapranov--Zelevinsky \cite[Chapter 5]{gkz}.  

\begin{art} \rm \label{projective toric schemes}
Given $R \in \zdop_+$, we choose projective coordinates on the projective space $\pdop^R_{\kcirc}$.
Let $A=(u_0,\ldots, u_R)\in M^{R+1}$ and ${\bf y}=(y_0:\cdots :y_R)\in \mb{P}^R(K)$. The {\it height function of ${\bf y}$} is defined as 
\[ a:\{0,...,R\} \ra \Gamma \cup \{ \infty\}, \quad j\mapsto a(j):=v(y_j).\]
The action of $\mb{T}$ on $\mb{P}^R_{K^\circ}$ is given by 
\[ (t,{\bf x})\mapsto (\chi^{u_0}(t)x_0:\cdots : \chi^{u_R}(t)x_R) .\]
We define the projective toric variety $\ms{Y}_{A,a}$ to be the closure of the orbit ${T} {\bf y}$. The generic fiber $Y_{A,a}$ is a toric variety respect to the torus $T/\tx{Stab}({\bf y})$. It follows from \cite[9.2]{gubler12} that $\ms{Y}_{A,a}$ is a $\mb{T}$-toric variety over $K^\circ$ with respect to the split torus over $K^\circ$ with generic fiber $T/\tx{Stab}({\bf y})$.

The {\it weight polytope}  $\tx{Wt}({\bf y})$ is defined as the convex hull of $A({\bf y}):=\{u_j\mid a(j)<\infty \}$. The {\it weight subdivision} $\tx{Wt}({\bf y},a)$ is the polytopal complex with support $\tx{Wt}({\bf y})$ obtained by projecting the faces of the convex hull of $\{ (u_j, \lambda_j) \in M_{\mb{R}}\times \mb{R}_+ \mid j=0,\ldots, R; \; \lambda_j \geq a(j)\}.$
We will see in the next result that the orbits of  $\ms{Y}_{A,a}$ can be read off from the weight subdivision.
\end{art}


\begin{prop}\label{projectivetoric}
There  is a bijective order preserving correspondence  between faces $Q$ of the weight subdivision $\tx{Wt}({\bf y},a)$ and 
 $\mb{T}$-orbits $Z$ of the special fiber of $\ms{Y}_{A,a}$ given by 

\[ Z = \{ {\bf x}\in (\ms{Y}_{A,a})_s\mid  x_j \neq 0 \Leftrightarrow u_j \in A({\bf y})\cap Q\}. \]
\end{prop}

\proof {This is the content of} \cite[Proposition 9.12]{gubler12}. \qed

\section{The cone of a normal affine toric variety}

We recall that $(K,v)$ is a valued field with valuation ring $K^\circ$, residue field $\widetilde{K}$ and value group $\Gamma\subset \mb{R}$. Let $\mb{T}=\te{Spec}(K^\circ[M])$ be the split torus over $\kcirc$ with generic fiber $T$. The free abelian group $M$ of rank $n$ is isomorphic to the character group of $T$. For an element $u\in M$, the corresponding character is denoted by $\chi^u$. Let $N=\te{Hom}(M,\mb{Z})$ be the dual abelian group of $M$. 

As we have seen in the previous section, a $\Gamma$-admissible cone $\sigma$ in $N_\rdop \times \rdop_+$ 
induces a normal affine $\mb{T}$-toric scheme $\ms{V}_\sigma =\tx{Spec}(K[M]^\sigma)$. This is  a $\mb{T}$-toric variety if and only if the vertices of $\sigma \cap (N_{\rdop}\times \{1\})$ are contained in $N_\Gamma \times \{1\}$ or if $v$ is discrete. 

In this section, we will show that every normal affine $\mb{T}$-toric variety $\ms{Y}=\te{Spec}(A)$ has this form proving Theorem \ref{cone}.
We may assume that the valuation is non-trivial as in the classical case of normal toric varieties over a field, the statement is well known (see \cite[ch. I, Theorem 1]{mumford73}). 
The $\tdop$-action induces an $M$-grading $A=\bigoplus_{m\in M} A_m$ on the $\kcirc$-algebra $A$. Since $T$ is an open dense orbit of $\ms{Y}$, we may and will assume that $A$ is a subalgebra of the quotient field $K(M)$ of $K[M]$. 

{The proof of Theorem \ref{cone} will be done in three steps along the classical proof dealing with the additional complications of a non-trivial value group and of a non-noetherian ring $A$. In Lemma \ref{saturated}, we will associate to the $M$-graded subalgebra $A$ a semigroup $S$ in $M\times \Gamma$. Then in a second step, we will show in Lemmata \ref{lemma2} and  \ref{lemma3} that $S$ gives rise to a $\Gamma$-admissible cone $\sigma$ in $N \times \rdop_+$. Finally, we will prove in Proposition \ref{cone proposition} that $\ms{Y}=\ms{V}_\sigma$ by using an approximation argument.}

\begin{lemma} \label{saturated}
The set  $S:=\{(m,v(a))\in M\times \Gamma\mid  a\chi^m\in A\backslash \{0\} \}$ is a saturated semigroup in $M\times \Gamma$.
\end{lemma}
\begin{proof}
Obviously, the set $S$ is a semigroup. Let $k(m,v(a))\in S$ for $m \in M$, $a \in K \setminus \{0\}$ and $k\in \mb{Z}_+ \setminus \{0\}$, i.e. $(a\chi^m)^k\in A$. 
By normality of $A$, we get $a\chi^m\in A$ and hence $(m,v(a))\in S$. 
\end{proof}


\begin{lemma}\label{lemma2}
There are $M$-homogeneous generators $a_1\chi^{m_1},\ldots ,a_k\chi^{m_k}$ of $A$. Moreover, the semigroup $S$ from Lemma \ref{saturated} and the set $\{ (0,1),(m_i,v(a_i))\mid i=1,\ldots, k\}$ generate the same cone in $M_\rdop \times \rdop$.
\end{lemma}

\begin{proof}
Since $\Ycal$ is a variety over $\kcirc$, it is clear that $A$ is a finitely generated $\kcirc$-algebra. Using that $A$ is an $M$-graded algebra, we find generators $a_1\chi^{m_1},\ldots ,a_k\chi^{m_k}$ of $A$. Obviously, every $(m_i,v(a_i))$ is contained in the cone generated by $S$ which we denote by $\te{cone}(S)$. Since the valuation $v$ is non-trivial, it is clear that $(0,1) \in \te{cone}(S)$.

It remains to show that the cone generated by $S$ is contained in the cone generated by $\{ (0,1),(m_i,v(a_i))\mid i=1,\ldots, k\}$. An element of $\te{cone}(S)$ is a finite sum $\sum_j\alpha_j (u_j,v(b_j))$ with $\alpha_j \in \rdop_+$ and $(u_j,v(b_j)) \in S$ {with $b_j\chi^{u_j}\in A$}. Using the above generators, we get
\[  b_j\chi^{u_j}=\lambda^{(j)}(a_1\chi^{m_1})^{l_1^{(j)}}\cdots (a_k\chi^{m_k})^{l_k^{(j)}}\quad \te{for } \lambda^{(j)}\in K^\circ \setminus \{0\}, l_1^{(j)},\ldots ,l_k^{(j)}\in \mb{Z}_+. \]
This implies 
\begin{eqnarray*}
v(b_j)&=&v(\lambda^{(j)})+\sum_{i=1}^kl_i^{(j)}v(a_i)\\
u_j&=&\sum_{i=1}^{k} l_i^{(j)}m_i.
\end{eqnarray*}
We conclude that the element $\sum_j\alpha_j (u_j,v(b_j))$ of $\te{cone}(S)$ is equal to
\begin{eqnarray*}
\sum_j \alpha_j \left( \sum_{i=1}^k l_i^{(j)}m_i,v(\lambda^{(j)})+\sum_{i=1}^k l_i^{(j)}v(a_i)\right)&=&\sum_j \alpha_j (0,v(\lambda^{(j)}))+\sum_j\sum_i \alpha_j l_i^{(j)}(m_i,v(a_i))\\
&=& (0,\lambda)+\sum_i \lambda_i(m_i,v(a_i))
\end{eqnarray*}
with $\lambda:=\sum_j \alpha_j v(\lambda^{(j)}) \in \rdop_+$ and $\lambda_i:=\sum_j \alpha_j l_i^{(j)}\in \rdop_+$. This proves the lemma.
\end{proof}
\begin{lemma}\label{lemma3}
The set $\sigma:=\{(\omega,s) \in N_\rdop \times \rdop \mid \langle u, \omega \rangle + ts \geq 0 \; \forall (u,t) \in S\}$  is a $\Gamma$-admissible cone in $N_\rdop \times \rdop_+$.
\end{lemma}
\begin{proof}
By definition, $\sigma$ is the dual cone of the cone generated by $S$. 
From Lemma \ref{lemma2}, we have 
\[ \sigma=\bigcap_{i=1}^k \left\{ (\omega, s)\in N_{\mb{R}}\times \mb{R}_+\mid \langle m_i,\omega \rangle+sv(a_i)\geq 0\right\}. \]
It remains to show that $\sigma$ doesn't contain a line. Suppose $\sigma$ contains a line. Then we have $\mb{R}\cdot (\omega,t) \subset \sigma$ for some $(\omega,t) \in N_{\mb{R}}\times \mb{R}_+$. Since $\sigma\subset N_{\mb{R}}\times \mb{R}_+$, we must have  $t=0$. Therefore the line is of the form $\mb{R}\cdot (\omega, 0)\subset N_{\mb{R}}\times \{0\}$.
For any $a\chi^u\in A \setminus \{0\}$, we have $(u, v(a))\in S$ and hence
\[ 0 \leq \langle (u,v(a)), (\lambda \omega,0) \rangle=\lambda \langle u,  \omega \rangle \quad \forall \lambda \in \mb{R}.\] 
This proves $u \in \omega^{\perp}$. Choosing a basis  $\{u_1,\ldots, u_n\}$ for $M$ such that $u_1,\ldots, u_{n-1}\in \omega^{\perp}$, we get $A\subset K[\chi^{\pm u_1},\ldots, \chi^{\pm u_{n-1}}]$.  On the other hand, $\Ycal$ is a $\tdop$-toric variety and hence the quotient field of $A$ is $K(\chi^{\pm u_1},\ldots, \chi^{\pm u_{n}})$. This is a contradiction and hence  $\sigma$ doesn't contain any line. We conclude that $\sigma$ is $\Gamma$-admissible.
\end{proof}


\begin{prop} \label{cone proposition}
Let $\Ycal =\Spec(A)$ be an affine normal $\tdop$-toric variety over $\kcirc$. Then $\Ycal=\Vcal_\sigma$ for the $\Gamma$-admissible cone $\sigma$ defined in Lemma \ref{lemma3}.
\end{prop}

\proof We have to show  $K[M]^\sigma=A$.  Take any $a\chi^m\in A \setminus \{0\}$. Since $(m, v(a))\in S$, we get  
$\langle m, \omega \rangle+t\cdot v(a)\geq 0$ for all $(\omega, t) \in \sigma$ and hence  
 $a\chi^m\in K[M]^\sigma$. This proves  $A\subset K[M]^\sigma$.

To prove the reverse inclusion, we take $a\chi^m\in K[M]^\sigma \setminus \{0\}$. By definition,  $(m,v(a))$ is in the dual cone $ \check{\sigma}$ of $\sigma$. Using biduality of convex polyhedral cones (see \cite[\S 1.2]{fulton93}), we conclude that $(m,v(a))$ is contained in the cone in $M_\rdop \times \rdop$ generated by $S$. By Lemma \ref{lemma2}, we get
\[ (m,v(a))= \kappa(0,1)+\sum_{i=1}^k\lambda_i (m_i,v(a_i)),\quad \kappa, \lambda_i \in \rdop_+.\]
From this, we deduce the following equivalent system of equations
\begin{eqnarray}
m&=&\sum_i \lambda_i m_i \label{exp}\\
v(a)&=&\kappa +\sum_i \lambda_i v(a_i). \label{coe} 
\end{eqnarray}
Now we show that it is always possible to choose all $\lambda_i \in \mb{Q}_+$. We may assume that  $\lambda_i > 0$ for all $i$, otherwise we omit these coefficients.
Let $b_1, \dots, b_s$ be a basis in $\mb{Q}^k$ for the solutions of the homogeneous equation associated to (\ref{exp}) and let $\mu \in \mb{Q}^k$ be a particular solution for (\ref{exp}). We will use the coordinates $(b_j^{(1)},\ldots, b_j^{(k)})$ for the vector $b_j$ of the basis.  The space of solutions $\mb{L}$ is given by
	\[ \mb{L}=\{  \mu +\sum_{j=1}^s \rho_j b_j \mid \rho_j\in \mb{R}, j=1,\ldots ,s\}. \]
	Since $\lambda:=(\lambda_i)\in \mb{R}_+^k$ is a solution of (\ref{exp}), there exist $\rho_j \in \mb{R}\; (j=1, \ldots, s)$ such that
	\[ \lambda=\mu +\sum_j \rho_j b_j\;.\]
	Now choose $\hat{\rho}_j\in \mb{Q}$ close to $\rho_j$, i.e.
	\[ \rho_j=\hat{\rho}_j+\epsilon_j\; \]
	with $|\epsilon_j|$ small. 
	Then 
	\[ \hat{\lambda}=\mu+\sum_j \hat{\rho}_j b_j\]
is also a solution of (\ref{exp}) in $\qdop^k$ which is close to $\lambda$. In particular, we may choose $|\ve_j|$ so small that all $\hat{\lambda}_i > 0$. Explicitly we have
	\[ \left( \begin{array}{c} \lambda_1 \\ \vdots \\ \lambda_k \end{array}\right)=\left( \begin{array}{c} \hat{\lambda}_1 \\ \vdots \\ \hat{\lambda}_k \end{array}\right)+\sum_{j=1}^s \epsilon_j b_j. \]
	Inserting this in (\ref{coe}), we get
	\begin{eqnarray*}
	 v(a)&=&\kappa +\sum_i \left( \hat{\lambda}_i+\sum_j \epsilon_jb_j^{(i)}\right)v(a_i)\\
	 &=&\kappa+\sum_i \hat{\lambda}_iv(a_i)+\sum_i\left( \sum_j \epsilon_jb_j^{(i)}\right)v(a_i).
	 \end{eqnarray*}
	 With $\alpha:=\sum_i\left( \sum_j \epsilon_jb_j^{(i)}\right)v(a_i)=\sum_j\epsilon_j\sum_i b_j^{(i)}v(a_i) $, we get
	 \[ v(a)= \kappa +\alpha +\sum_i \hat{\lambda}_iv(a_i).\]
	 It is easy to see that we may choose $\epsilon_1, \dots , \epsilon_s$ in a small neighbourhood of $0$ such that  $\kappa +\alpha\geq 0$. We conclude that it is possible to choose the coefficients in \eqref{exp} rational and we have  
\begin{equation} \label{rational repr}
(m,v(a))= (\kappa + \alpha)(0,1)+\sum_{i=1}^k\hat{\lambda}_i (m_i,v(a_i)),\quad \kappa + \alpha \in \rdop_+, \hat{\lambda}_i \in \qdop_+.
\end{equation}

The above shows that $(m,v(a))=(0,\kappa)+\sum_i \lambda_i (m_i,v(a_i))$ with $\lambda_i\in \mb{Q}_+$ and $\kappa \in \mb{R}_+$. Let $R$ be a positive integer such that  $R\lambda_i \in \mb{Z}_+$ for  $i=1,\ldots ,k$. Then we get 
\[ R(m,v(a))=R(0,\kappa)+\sum_i R\lambda_i(m_i,v(a_i)).\] 
This proves in particular that $R\kappa\in \Gamma$. Since $(0,R\kappa), (m_i,v(a_i))\in S \; (i=1,\ldots, k)$ and $R\lambda_i\in \mb{Z}_+$, we conclude that  $(Rm, Rv(a))$ is also in the semigroup  $S$. {This means that there is a $b \in K$ with $v(b)=v(a^R)$ such that $b\chi^{Rm} \in A$. Since $a^{R}=ub$ for a unit in $\kcirc$, we get} $(a\chi^m)^R\in A$. By normality of $A$, this implies that $a\chi^m\in A$. We conclude that $K[M]^\sigma=A$ and hence $\ms{Y}=\ms{V}_\sigma$.  \qed

\begin{proof}[Proof of Theorem \ref{cone}] 
We assume that $v$ is not a discrete valuation. We have seen in Propositions \ref{cone proposition} and \ref{finitepresentation} that the map $\sigma \mapsto \Vcal_\sigma$ from the set of those  $\Gamma$-admissible cones in $N_\rdop \times \rdop_+$ for which  the vertices of $\sigma \cap (N_\rdop \times \{1\})$ are contained in $N_\Gamma \times \{1\}$ to the set of isomorphism classes of affine normal  $\tdop$-toric varieties over $\kcirc$ is surjective.  By \cite[Proposition 6.24]{gubler12}, we can reconstruct the cone $\sigma$  from the $\tdop$-toric scheme $\Vcal_\sigma$ by applying the tropicalization map to the set of integral points of $T \cap \Vcal_\sigma$ and hence the correspondence is indeed bijective. 
If $v$ is a discrete valuation, then the same argument works if we omit the additional condition on the vertices of the cones.
\end{proof}

\section{Construction of the Cartier divisor} \label{Construction of the Cartier divisor}

Let $(K,v)$ be a valued field with valuation ring $\kcirc$, value group $\Gamma=v(K^\times)\subset \rdop$ and residue field $\ktilde$. Let $\tdop$ be the split torus of rank $n$ over $\kcirc$. In this section, we consider a non-empty affine open subset $\Ucal_0$ in a normal $\tdop$-toric variety $\Ycal$ over $\kcirc$. 

{The main goal is to construct an effective Cartier divisor $D$ on the smallest $\tdop$-invariant open subset $\Ucal$ of $\Ycal$ containing $\Ucal_0$  with support equal to  $\Ucal \setminus \Ucal_0$. This will be achieved in Proposition \ref{Cartier divisor to Dcal}. We will start by noting in Proposition \ref{complement is Weil divisor} that normality of $\Ycal$ yields that the complement of $\Ucal_0$ is a Weil divisor. Using the divisorial intersection theory from  \S \ref{divisors on varieties}, we will deduce in Proposition \ref{linear equivalence of translates} that all translates  of this Weil divisor  are rationally equivalent on $\Ucal$. This will be enough to deduce that the Weil divisor is actually a Cartier divisor on $\Ucal$ proving Proposition \ref{Cartier divisor to Dcal}. In fact, it will also follow in Corollary \ref{Cartier divisor and translation} that $\Ocal(D)$ is generated by global sections. This will be important in the next section, where we will show that $D$ is ample and gives rise to an equivariant immersion of $\Ucal$ to a projective space with a linear $\tdop$-action. This will allow us in \S \ref{Proof of Sumihiro} to reduce the proof of Sumihiro's theorem to an easier projective variant.}

\begin{prop} \label{complement is Weil divisor}
Let $\Ucal_0$ be a non-empty affine open subset of a normal variety $\Ycal$ over $\kcirc$. Then every irreducible component of $\Ycal \setminus \Ucal_0$ has codimension $1$ in $\Ycal$. 
\end{prop}

\proof By removing the irreducible components of $\Ycal \setminus \Ucal_0$ of codimension $1$, we may assume that $\Ycal \setminus \Ucal_0$ has no irreducible components of codimension $1$. Then we have to prove $\Ucal_0 = \Ycal$. We may assume that $\Ycal$ is an affine variety $\Spec(A)$.  Using Proposition \ref{intersection of local rings}(d), we get $\Ocal(\Ucal_0)=\Ocal(\Ycal)$ and hence the affine varieties $\Ucal_0$ and $\Ycal$ are equal. \qed

\vspace{2mm}

In the following result, we will use the notions introduced in Section \ref{divisors on varieties}.

\begin{prop} \label{surjectivity of pullback}
Let $p_2$ be the canonical projection of $\tdop \times_{K^\circ} \Ycal$ onto the variety $\Ycal$ over $\kcirc$ and let $\ms{D}$ be a cycle of codimension $1$ in $\tdop \times_{K^\circ} \Ycal$. Then there is a cycle $\ms{D}'$ on $\Ycal$ of codimension $1$  such that $p_2^*(\ms{D}')$ is rationally equivalent to  $\ms{D}$.
\end{prop}

\proof We note that  irreducible components of $(\tdop \times_{K^\circ} \Ycal)_s$ are given by $\tdop_s \times_{\widetilde{K}} V$ with $V$ ranging over the irreducible components of $\Ycal_s$. We conclude that every vertical cycle of codimension $1$ in $\tdop \times_{K^\circ} \Ycal$ is the pull-back of a vertical cycle of codimension $1$ in $\Ycal$. This reduces the claim to the horizontal parts where it is a standard fact from algebraic intersection theory on varieties over a field \cite[Proposition 1.9]{fulton98}. \qed

\begin{art} \label{pullback by it} \rm 
 {The generic fibre of $\tdop$ is denoted by $T$. Let  $T^\circ$ be the affinoid torus in $\Tan$ given by $\{x \in \Tan \mid |x_1(x)|= \dots = |x_n(x)|=1\}$} in terms of torus coordinates $x_1, \dots, x_n$ and let $\Ycal$ be a variety over $\kcirc$ with generic fibre $Y$. 
For $t \in T^\circ(K)$, the reduction $\tilde{t} \in \tdop_s(\ktilde)$ is well-defined. Let $i_t: \Ycal \rightarrow \tdop \times_{K^\circ} \Ycal$ be the embedding over $\Ycal$ induced by the integral point  of $\tdop$ corresponding to $t$. We are going to define the pull-back $i_t^*(\Zcal)$ for every cycle $\Zcal$ on $\tdop \times_{K^\circ} \Ycal$ which satisfies the following {\it flatness condition}: We assume that every component of the horizontal (resp. vertical) part of $\Zcal$ is flat over $T$ (resp. $\tdop_s$).  

Since $i_t$ induces a regular embedding of $Y$ into $T \times_{K} Y$ (resp. of $\Ycal_s$ into $\tdop_s \times_{\widetilde{K}} \Ycal_s$), the pull-back of the horizontal part (resp. vertical part) of $\Zcal$ is a well-defined cycle  on $Y$ (resp. $\Ycal_s$) (see \cite[Chapter 6]{fulton98}). We define $i_t^*(\Zcal)$ as the sum of these two pull-backs. Clearly, this pull-back keeps the codimension and is linear in $\Zcal$.  
\end{art}

\begin{art} \label{interpretation as divisoral intersection} \rm 
For $t \in T^\circ(K)$ with coordinates $t_1:=x_1(t), \dots , t_n:=x_n(t)$, let $D_{t_j}$ be the Cartier divisor on $\tdop \times_{K^\circ} \Ycal$ given by pull-back of $\Div(x_j-t_j)$ with respect to the canonical projection onto  $\tdop = (\Tor)_\kcirc$. Let $\Zcal$ be a cycle on $\tdop \times_{K^\circ} \Ycal$ satisfying the  flatness condition from \ref{pullback by it}. Then we may use the proper intersection product with Cartier divisors from \ref{proper intersection with cycles}  to get 
$$(i_t)_*(i_t^*(\Zcal))= D_{t_1} \dots D_{t_n}.\Zcal.$$
Indeed,   the  flatness condition  ensures that the right hand side is a proper intersection product and hence the claim follows from \cite[Example 6.5.1]{fulton98}. By Proposition \ref{commutativity}, the proper intersection product on the right  is symmetric with respect to the Cartier divisors.
\end{art}

Let $\Ycal, \Ycal'$ be varieties over $\kcirc$ and let $\varphi:\tdop \times_{K^\circ} \Ycal \rightarrow \tdop \times_{K^\circ} \Ycal'$ be a flat morphism over $\tdop$. The point $t \in T^\circ(K)$ corresponds to an integral point of $\tdop$  inducing a flat morphism $\varphi_t:\Ycal \rightarrow \Ycal'$ by base change from $\varphi$. 
We recall from \ref{flat pullback} that we have also introduced the pull-back with respect to flat morphisms. The following result shows some functoriality with the above pull-backs. We will use the canonical projection $p_2:\tdop\times_{K^\circ} \Ycal' \rightarrow \Ycal'$. 

\begin{prop} \label{functoriality}
Under the hypothesis above, let $\Zcal'$ be a cycle of $\Ycal'$. Then the cycle $\varphi^*(p_2^*(\Zcal'))$ satisfies the flatness condition from \ref{pullback by it} and we have $i_t^*(\varphi^*(p_2^*(\Zcal')))=\varphi_t^*(\Zcal')$.
\end{prop}

\proof {The} cycle $p_2^*(\Zcal)$ satisfies the flatness condition. Using that $\varphi$ is a flat morphism over $\tdop$, we deduce that  $\varphi^*(p_2^*(\Zcal'))$ also fulfills the  flatness condition . Since the pull-backs are defined for horizontal and vertical parts in terms of the corresponding operations for varieties over fields, the claim follows from  \cite[Proposition 6.5]{fulton98}, {where the functoriality for pull-backs of cycles of varieties has been proved}. \qed

\begin{lemma} \label{it and linear equivalence}
Let $\Ycal$ be a variety over $\kcirc$ and let $t \in T^\circ(K)$. Suppose that $g$ is a rational function on $\tdop \times_{K^\circ} \Ycal$ such that every irreducible component of  the support of the restriction of $\Div(g)$ to the generic fibre is flat over $T$.  Then $g(t,\cdot)$ is a rational function on $\Ycal$ and we have $i_t^*(\cyc(\Div(g)))=\cyc(\Div(g(t,\cdot)))$.
\end{lemma}

\proof The  flatness assumption yields the first claim immediately. Note that  the  vertical components of $\cyc(\Div(g))$ are  automatically flat over $\tdop_s$ and hence we get a well-defined cycle $i_t^*(\cyc(\Div(g)))$  on $\Ycal$. The second claim follows easily from the fact that we may write $i_t^*$ as an $n$-fold proper intersection product with Cartier divisors (see \ref{interpretation as divisoral intersection}) and from Proposition \ref{commutativity}. \qed

\begin{prop} \label{first Chow group isomorphism}
Let $\Ycal$ be a variety over $\kcirc$. Then pull-back with respect to the canonical projection $p_2: \tdop \times_{K^\circ} \Ycal \rightarrow \Ycal$ induces an isomorphism $p_2^*: CH^1(\Ycal) \rightarrow CH^1(\tdop \times_{K^\circ} \Ycal)$.
\end{prop}

\proof By \ref{equivalence}, $p_2^*$ is compatible with rational equivalence and hence it is well-defined on the Chow groups. Surjectivity follows from Proposition \ref{surjectivity of pullback}. Suppose that $\Dcal$ is a cycle of codimension $1$ on $\Ycal$ such that $p_2^*(\Dcal)$ is rationally equivalent to $0$ on $\tdop \times_{K^\circ} \Ycal$. Using  Lemma \ref{it and linear equivalence} for the unit element $e$ in $T^\circ(K)$, we deduce that $\Dcal$ is rationally equivalent to $0$. This proves injectivity. \qed

\vspace{2mm} 

We have a similar statement for Picard group as pointed out by Qing Liu and C. P\'epin. 

\begin{prop} \label{Picard group isomorphism}
Let $\Ycal$ be a normal variety over $\kcirc$. Then pull-back with respect to $p_2$ induces an isomorphism $\Pic(\Ycal) \rightarrow \Pic(\tdop \times_{K^\circ} \Ycal)$. 
\end{prop}

\proof {This statement was proved in \cite[Remark 9.6]{gubler12} based on an argument of Qing Liu and C. P\'epin.} \qed

\vspace{2mm}
Now let $\Ycal$ be a normal $\tdop$-toric variety  over $\kcirc$ and let $\Dcal$ be a cycle of codimension $1$ in $\Ycal$.  Note that $t \in T^\circ(K)$ acts on $\Ycal$ and we denote by $\Dcal^t$ the pull-back of $\Dcal$ with respect to this flat morphism.

\begin{prop} \label{linear equivalence of translates}
Under the hypothesis above,  $\Dcal^t$ is rationally equivalent to $\Dcal$.  
\end{prop}

\proof Let $\sigma: \tdop \times_{K^\circ} \Ycal \rightarrow \Ycal$ be  the torus action on $\Ycal$. It follows from Propositions \ref{surjectivity of pullback} that $\sigma^*(\Dcal)$ is rationally equivalent to $p_2^*(\Dcal')$ for a cycle $\Dcal'$ of codimension $1$ in $\Ycal$. Then Lemma \ref{it and linear equivalence} and Proposition  \ref{functoriality}, applied for the unit element $e$, show that $\Dcal$ is rationally equivalent to $\Dcal'$. By  \ref{equivalence}, we conclude that 
$\sigma^*(\ms{D})$ is rationally equivalent to $p_2^*(\ms{D})$. If we apply Proposition \ref{functoriality} again, but now in $t$ instead  of $e$,  we get the claim. \qed


\begin{lemma} \label{T-invariance of U}
Let $\Ucal_0$ be a non-empty open subset of the  $\tdop$-toric variety $\Ycal$ over $\kcirc$ and let $\Ucal := \bigcup_{t \in T^\circ(K)} t\Ucal_0$. Then $\ms{U}$ is the smallest $\mb{T}$-invariant (open) subset containing $\Ucal_0$.
\end{lemma}
\begin{proof}
Consider the subset $S$ of $\mb{T}$ such that translation with its elements leaves $\ms{U}$ invariant. The subset $S \cap \tdop_s$ is equal to the stabilizer of  $\Ycal_s \setminus \ms{U}_s$ and hence it is an algebraic subgroup of $\mb{T}_s$. By construction, it contains $\tdop_s(\ktilde)$ and hence it is equal to $\mb{T}_s$. We use the same argument for the points of $S$ contained in the generic fibre $T=\tdop_\eta$. Again, $S \cap T$ is an algebraic subgroup containing  $T^\circ(K)$.  Since $T^\circ$ is an $n$-dimensional affinoid torus, we conclude that $T^\circ(K)$ is Zariski dense in $T$ and hence the algebraic subgroup is the torus $T$ over $K$. We conclude that $\ms{U}$ is $\mb{T}$-invariant. This proves the claim immediately. 
\end{proof}

\begin{art} \label{choice of Dcal} \rm
Since the torus $\tdop_s$ acts continuously on the discrete set of the generic points of $\Ycal_s$, every such generic point is fixed under the action. We conclude that every irreducible component of $\Ycal_s$ is invariant under the $\tdop$-action. This means that the special fibres of $\Ucal$ and $\Ucal_0$ have the same generic points. We have seen in Proposition \ref{complement is Weil divisor} that $\Ucal \setminus \Ucal_0$ is a union of irreducible components of codimension $1$ and hence every such irreducible component is horizontal. Let $\Dcal$ be the horizontal cycle on $\Ucal$ given by the formal sum of these irreducible components. 
\end{art}

\begin{prop} \label{Cartier divisor to Dcal}
Under the hypothesis above, there is a unique Cartier divisor $D$ on $\Ucal$ such that $\Dcal = \cyc(D)$. Moreover, this Cartier divisor is effective.
\end{prop}

\proof For $t \in T^\circ(K)$,  Proposition \ref{linear equivalence of translates} yields a non-zero  rational function $f_t$ on $\Ucal$ such that $\Dcal - \Dcal^t = \cyc(\Div(f_t))$. Since $\Ucal \setminus t^{-1} \Ucal_0$ is equal to the support of $\Dcal^t$, we deduce that the restriction of $\Dcal$ to $t^{-1} \Ucal_0$ is the Weil divisor given by the rational function $f_t$ on $t^{-1} \Ucal_0$.   By Corollary \ref{equality of supports}, the Cartier divisor on a normal variety is uniquely determined by its associated Weil divisor. This yields immediately that $\{(t^{-1} \Ucal_0, f_t)\mid t \in T^\circ(K)\}$ is a Cartier divisor on $\Ucal$ with associated Weil divisor $\Dcal$ and uniqueness follows as well. By Corollary \ref{equality of supports}, the Cartier divisor $D$ is effective.
\qed

\begin{prop} \label{Cartier divisor and multiplication}
Let $\sigma:\tdop \times_{K^\circ} \Ycal \rightarrow \Ycal$ be the torus action of the normal $\tdop$-toric variety $\Ycal$ over $\kcirc$ and let $D$ be  the Cartier divisor from Proposition \ref{Cartier divisor to Dcal}. Then $\sigma^*(D)$ is linearly equivalent to $p_2^*(D)$. 
\end{prop}

\proof The unit element $e$ in $T^\circ(K)$ induces the section $i_e$ of $\sigma$ and $p_2$. Then the claim follows from Proposition \ref{Picard group isomorphism}. Another way to deduce the claim is to use the corresponding statement for cycles of codimension $1$ (see Proposition \ref{first Chow group isomorphism}) together with Corollary \ref{equality of supports}. \qed

\begin{corollary} \label{Cartier divisor and translation}
Let $D$ be the Cartier divisor from Proposition \ref{Cartier divisor to Dcal} and let $D^t$ be its pull-back with respect to translation by $t \in T^\circ(K)$. Then the invertible sheaves $\Ocal(D^t)$ and $\Ocal(D)$ on $\Ucal$ are isomorphic. Moreover, $\Ocal(D)$ is generated by global sections.
\end{corollary}

\begin{proof} 
The first claim  follows from Proposition \ref{Cartier divisor and multiplication} by applying $i_t^*$. By Proposition \ref{Cartier divisor to Dcal}, $s_{D^t}$ is a global section with support $\Ucal \setminus t^{-1} \Ucal_0$ and hence the second claim follows from the first.
\end{proof}


\section{Linearization and  immersion into projective space} \label{linea}

Let $\tdop$ be the split torus of rank $n$ over $\kcirc$ and let $\Ycal$ be a normal $\tdop$-toric variety over $\kcirc$. We denote by $\mu:\tdop \times_{K^\circ} \tdop \rightarrow \tdop$ the multiplication map, by $\sigma:\tdop \times_{K^\circ} \Ycal \rightarrow \Ycal$ the group action and by $p_2: \tdop \times_{K^\circ} \Ycal \rightarrow \Ycal$ the second projection. As in the previous section, we consider a non-empty affine open subset $\Ucal_0$ of $\Ycal$ and the smallest $\tdop$-invariant open subset $\Ucal$ of $\Ycal$ containing $\Ucal_0$. In Proposition \ref{Cartier divisor to Dcal}, we have constructed an effective Cartier divisor $D$ on $\Ucal$ with $\supp(D)=\Ucal \setminus \Ucal_0$ such that $\cyc(D)$ is a horizontal cycle with all multiplicities equal to $1$. {In this section, we will see  that $\Ocal(D)$ has a $\tdop$-linearization (Proposition \ref{existence of linearization}) and is ample (Proposition \ref{ampleness}) leading in \ref{construction of an invariant immersion} to a $\tdop$-equivariant immersion into a projective space with linear $\tdop$-action. 
This will be summarized in Proposition \ref{the A-structure} which reduces the proof of Sumihiro's theorem to an easier projective variant shown in the next section.}

\begin{definition} \label{definition of linearization} 
First, we recall the definition of a { \it $\tdop$-linearization} of a line bundle $L$ on a toric variety (see \cite{mumford94} for details). Geometrically, a $\tdop$-linearization is a lift of the torus action on $\ms{Y}$ to an action on $L$ such that the zero section is $\mb{T}$-invariant. In terms of the underlying invertible sheaf $\ms{L}$, a $\tdop$-linearization is an isomorphism 
\[ \phi: \sigma^* \ms{L} \ra p^*_2 \ms{L}, \]
of sheaves on $\mb{T}\times_{K^\circ} \ms{Y}$  
satisfying the cocycle condition
\begin{equation}\label{comp}
p_{23}^*\phi \circ (\id_{\mb{T}}\times \sigma)^*\phi=(\mu \times \id_{\ms{Y}})^*\phi,
\end{equation}
where $p_{23}: \mb{T} \times_{K^\circ} \mb{T} \times_{K^\circ} \ms{Y} \ra \mb{T} \times_{K^\circ} \ms{Y}$ is the projection to the last two factors.
\end{definition}

We need the following {application} of a result of Rosenlicht.

\begin{lemma} \label{Rosenlicht}
For every $f \in \Ocal(\tdop \times_{K^\circ} \Ucal)^\times$, there is a character $\chi$ on $T$ and a $g \in \Ocal(\Ucal)^\times$  such that $f=\chi \cdot g$. 
\end{lemma}

\proof Note that $\Ocal(\tdop)^\times$ is the set of characters on $T$ multiplied by units in $\kcirc$. Then the claim follows from \cite[Theorem 2]{rosenlicht}, {where it is proved in the case of fields}. \qed

\begin{prop} \label{existence of linearization}
The invertible sheaf $\Ocal(D)$ has a $\tdop$-linearization.
\end{prop}

\proof  By Proposition \ref{Cartier divisor and multiplication}, we have an isomorphism
\[ \phi: \sigma^*\Lcal\ra p_2^*\Lcal\] 
for the invertible sheaf $\Lcal := \ms{O}(D)$. Both sides of \eqref{comp} are isomorphisms between the same invertible sheaves on $\tdop \times_{K^\circ} \tdop \times_{K^\circ} \Ucal$ and hence there is a unique $f \in \Ocal(\tdop \times_{K^\circ} \tdop \times_{K^\circ} \Ucal)^\times$ such that that the left hand side is obtained by multiplying the right hand side with  $f$. 
We may choose  $\phi$ such that we have  the canonical isomorphism over $\{e\} \times \Ucal$ and hence we get  
$f(e,\cdot, \cdot)=1$ and $
f(\cdot,e,\cdot)=1$. By Lemma \ref{Rosenlicht}, there are characters $\chi_1,\chi_2$ on $T$ and $g \in \Ocal(\Ucal)^\times$ such that
$f(t_1,t_2,u)=\chi_1(t_1) \chi_2(t_2) g(u)$ for all $t_1,t_2 \in T(\overline{K})$ and $u \in \Ucal(\overline{K})$. 
Since $f(e,e,u)=1$, we get $g=1$. Therefore \[ f(t_1,t_2,u)=\chi_1(t_1)\chi_2(t_2)=f(t_1,e,u)f(e,t_2,u)   =1.\]                                     
By density of the $\overline{K}$-rational points, we get $f=1$ and  \eqref{comp} holds. \qed

\begin{art} \rm \label{action on global section}
The $\tdop$-linearization on $\ms{L}=\ms{O}(D)$ induces a dual action of $\tdop$ on the space $H^0(\ms{U},\mc{L})$ of global sections, given by 
the composition $\hat{\sigma}$ of the canonical $\kcirc$-linear maps
\[ H^0(\ms{U},\ms{L})\ra H^0(\mb{T}\times_{K^\circ} \ms{U},\sigma^*\ms{L})\ra H^0(\mb{T}\times_{K^\circ} \ms{U},p_2^*\ms{L})\ra H^0(\mb{T},\ms{O}_{\mb{T}})\otimes_{K^\circ} H^0(\ms{U},\ms{L}), \]
where the last isomorphism comes from the K\"unneth formula (see \cite{kempf}). We refer to \cite[Chapter 1, Definition 1.2]{mumford94} for the definition of a dual action. This was written for vector spaces over a base field, but the same definition applies in case of a free $\kcirc$-module. Since $V:=H^0(\ms{U},\ms{L})$ is a torsion free $\kcirc$-module, $V$ is indeed free (see \cite[Lemma 4.2]{gubler12} {for a proof}). A dual action means that the torus $\tdop$ acts linearly on the possibly infinite dimensional projective space $\pdop(V)={\rm Proj}(\kcirc[V])$. 

The dual action $\hat{\sigma}$ induces an action of  $t \in T^\circ(K)$ on $V$ which we denote by $s \mapsto t \cdot s$. For $s \in V=H^0(\ms{U},\ms{L})$, the action is geometrically given by $(t \cdot s)(u) = t^{-1}(s(tu))$, $u \in \Ucal$, where $t^{-1}$ operates on the underlying line bundle using the linearization. 
\end{art}

\begin{lemma}\label{sections}
Let $x_1,\ldots, x_k$ be affine coordinates of $\ms{U}_0$ considered as rational functions on $\Ucal$. Then there exists $\ell\in \mb{Z}_+$ such that for every $i \in \{1, \dots, k\}$, the meromorphic section $s_i:= x_i s_{\ell D}$ of $\ms{O}(\ell D)$ is in fact a global section. {Here, $s_{\ell D}$ denotes the canonical global section of $\ms{O}(\ell D)$.}
\end{lemma}
\begin{proof}
Using the theory of divisors from \S \ref{divisors on varieties}, we get the identity
\[\cyc(\Div(x_i))=\sum_jm_{ij}Z_j+\ms{V}\]
of cycles on $\Ucal$, where $Z_j$ are the irreducible components of $\ms{U}\backslash \ms{U}_0$ and where $\ms{V}$ is an effective cycle of codimension $1$ in $\Ucal$ which meets $\ms{U}_0$. 
By construction (see Proposition \ref{Cartier divisor to Dcal}), we have  $\cyc(D)=\ms{D}= \sum_j Z_j$. For $\ell :=-\min_{ij}\{ m_{ij},0\}$, we get 
\[ \cyc(\Div(x_i))+\ell\ms{D}=\sum_j m_{ij} Z_j +\ell\ms{D}+\ms{V} = \sum_j(m_{ij}+\ell)Z_j+\ms{V}\geq 0 .\]
Therefore the Weil divisor $\Div(x_i)+ \ell D$ is effective. By Corollary \ref{equality of supports}, we conclude that $x_i s_{\ell D}$ is 
a global section of $\ms{O}(\ell D)$. 
\end{proof}

\begin{art} \rm \label{construction of an immersion}
By Proposition \ref{noetherian topological space}, $\ms{Y}$ is a noetherian topological space and therefore  $\ms{U}$ is quasicompact. Using Lemma \ref{T-invariance of U}, there is a finite subset $S$ of  $T^\circ (K)$ such that $\ms{U}=\bigcup_{t \in S} t^{-1} \ms{U}_0$. We have seen in Lemma \ref{sections} that the affine coordinates  $x_1,\ldots ,x_k$ of $\Ucal_0$ induce global sections $s_1, \dots ,s_k$ of $\ms{O}(\ell D)$. Then the dual action from \ref{action on global section} gives global sections $t\cdot s_1,\ldots ,t\cdot s_k$  of $\ms{O}(\ell D)$ induced by affine coordinates of $t^{-1} \Ucal_0$. We conclude that $(t \cdot s_j)_{t \in S, j=1,\dots, k}$ generate $\ms{O}(\ell D)$. By construction, the global section $t \cdot s_D$ has support $\Ucal \setminus t^{-1} \Ucal_0$.  We get a morphism 
\[ \psi: \Ucal \rightarrow \pdop^{R'}_\kcirc, u\mapsto (\cdots : t\cdot s_1(u):\cdots : t\cdot s_k(u):t\cdot s_D^\ell(u):\cdots)_{ t\in S}\] 
with $R':=|S|(k+1)-1$. Note that this map is well defined because $\ms{O}(\ell D)$ is generated by these global sections, and we have $\psi^*(\ms{O}_{\pdop^{R'}_\kcirc}(1))\simeq \ms{O}(\ell D)$. 
\end{art}

\begin{prop} \label{ampleness}
The morphism $\psi$ is an immersion and hence $\Lcal$ is ample.
\end{prop}

\proof For $t \in S$, the support of the Cartier divisor $\Div(t \cdot s_D)=D^t$ is equal to $\Ucal \setminus t^{-1}\Ucal_0$. Let $y_j$ be the coordinate of $\pdop^{R'}_\kcirc$ corresponding to $t \cdot s_D^\ell$ with respect to the morphism $\psi$. Then we get $\psi^{-1}\{y_j \neq 0\} = t^{-1} \Ucal_0$. Since $t\cdot x_1, \dots , t \cdot x_k$ are affine coordinates on $t^{-1} \Ucal_0$, we conclude easily that $\psi$ restricts to a closed immersion of $t^{-1}\Ucal_0$ into the open subvariety $\{y_j \neq 0\}$ of $\pdop^{R'}_\kcirc$. Since these open subvarieties form an open covering of  $\pdop^{R'}_\kcirc$, we may use  \cite[Corollaire 4.2.4]{EGAI} to conclude that the morphism $\psi$ is an immersion and hence $\ms{O}(D)$ is ample.\qed

\begin{art} \rm \label{construction of an invariant immersion}
Let $V_{\ell0}$ be the submodule of $V_\ell:= H^0(\Ucal, \Lcal^\ell)$ which is generated by the global sections $(t\cdot s_j)_{t \in S, j=1,\dots, k}$ and $(t \cdot s_D^\ell)_{t \in S}$ used in the definition of $\psi$ in \ref{construction of an immersion}. Since $\ms{O}(\ell D)$ has a $\tdop$-linearization, we get a dual action $\hat{\sigma}$ of $\tdop$ on $V_\ell$ similarly {to} \ref{action on global section}. 

A $\kcirc$-submodule $W$ of $V_\ell$ is called {\it invariant} under the dual action of $\tdop$ if $\hat{\sigma}(W) \subset  A \otimes_{K^\circ} W$ for $A:=H^0(\mb{T},\ms{O}_{\mb{T}})= \kcirc[M]$. The lemma on p. 25 of \cite{mumford94} generalizes {in a straightforward manner} to our setting and hence there is a finitely generated submodule $W$ of $V_\ell$ which is invariant under the dual action of $\tdop$ and contains $V_{\ell0}$. Since $W$ is $\kcirc$-torsion free, we conclude that $W$ is a free $\kcirc$-module of finite rank $R+1$. 

We get a morphism $i: \Ucal \rightarrow \pdop(W)$ with $i^*(\ms{O}_{\pdop(W)}(1))\cong \ms{O}(\ell D)$. The dual action of $\tdop$ on $W$ induces a linear action of $\tdop$ on the projective space $\pdop(W)$. By construction, $i$ is $\tdop$-equivariant. Since $i$ factorizes through  $\psi$, we deduce from Proposition \ref{ampleness} that $i$ is an immersion. 
\end{art}

Recall that $\tdop = \Spec(\kcirc[M])$ is the split torus of rank $n$. We summarize our findings:

\begin{prop} \label{the A-structure}
Let $\Ucal_0$ be a non-empty affine open subset of the normal $\tdop$-toric variety $\Ycal$ over $\kcirc$ and let $\Ucal$ be the smallest $\tdop$-invariant open subset of $\Ycal$ containing $\Ucal_0$. Then there is a $\tdop$-equivariant open immersion of $\Ucal$ into a projective $\tdop$-toric variety $\Ycal_{A,a}$ given by $A \in M^{R+1}$ and height function $a$ as in \ref{projective toric schemes}.
\end{prop}

\proof Let $i: \Ucal \rightarrow \pdop(W)$  be the $\tdop$-equivariant immersion from \ref{construction of an invariant immersion}. Then the closure  $\Ycal$  of $i(\Ucal)$ in $\pdop(W)$ is a projective $\tdop$-toric variety over $\kcirc$ on which $\tdop$-acts linearly. We choose a $K$-rational point $\yb$ in the open dense orbit of $i(\Ucal)$. By \cite[Proposition 9.8]{gubler12}, there are suitable coordinates on $\pdop(W)$ and $A \in M^{R+1}$ such that $\Ycal = \Ycal_{A,a}$ for the height function $a$ of $\yb$ defined in \ref{projective toric schemes}. \qed

\section{Proof of Sumihiro's theorem} \label{Proof of Sumihiro}

{In this section, we will finally prove Sumihiro's theorem for normal toric varieties over $\kcirc$ as announced in Theorem \ref{sumihiro1}.}
Sumihiro's theorem is wrong for arbitrary non-normal toric varieties even over a field (see \cite[Example 3.A.1]{cox}  for a projective counterexample). However, we will show {first} in this section that Sumihiro's theorem holds for open invariant subsets of projective toric varieties over $\kcirc$ with a {\it linear} torus action. {Note that such projective toric varieties are not necessarily normal. As a consequence of Proposition \ref{the A-structure}, we  will obtain Sumihiro's theorem for normal toric varieties over $\kcirc$. At the end, Theorem \ref{sumihiro2} will easily follow from Sumihiro's theorem similarly as in the classical case of normal toric varieties over a field.} 

In this section, $(K,v)$ will be a valued field with value group $\Gamma = v(K^\times) \subset \rdop$. Moreover, $\tdop=\Spec(\kcirc[M])$ is a split torus over the valuation ring $\kcirc$ of rank $n$ and we consider a projective $\tdop$-toric variety over $\kcirc$ with a linear $\tdop$-action. By \cite[Proposition 9.8]{gubler12} the latter is a toric subvariety $\Ycal_{A,a}$ of $\pdop_\kcirc^{R}$ as in \ref{projective toric schemes}  for $A \in M^{R+1}$, height function $a$ and suitable projective coordinates $x_0, \dots , x_R$.


\begin{art} \rm \label{setting for quasiprojective Sumihiro}
We fix a point $z \in \Ycal_{A,a}$ and a closed $\tdop$-invariant subset $Y$ of $\Ycal_{A,a}$ with $z \not \in Y$. Since $Y$ is a closed subset of the ambient projective space $\pdop_\kcirc^{R}$, there is a $k \in \zdop_+$ and $s_0 \in H^0(\pdop_\kcirc^{R}, \Ocal(k))$ such that $s_0|_Y=0$ and $s_0(z)\neq 0$.

Obviously, $V:= H^0(\pdop_\kcirc^{R}, \Ocal(k))$ is a free $\kcirc$-module of finite rank. The linear $\tdop$-action on  $\pdop_\kcirc^{R}$ induces a linear representation of $\tdop$ on $V$, i.e. a homomorphism $S: \tdop \rightarrow GL(V)$ of group schemes over $\kcirc$. We say that $s \in V$ is {\it semi-invariant} if there is $u \in M$ such that $S_t(s)= \chi^u(t)s$ for every $t \in \tdop$ and for the character $\chi^u$ of $\tdop$ associated to $u$. In  the following, the $\kcirc$-submodule 
$$W:=\{s\in H^0(\mb{P}^R_{K^\circ}, \ms{O}(k))\mid \exists \lambda \in K^\circ \backslash \{0\} \te{ s. t. } \lambda s|_Y=0 \}$$
of $V$ will be of interest. Note that $W$ is equal to the set of global sections $s$ of $\ms{O}(k)$ which vanish on the generic fiber $Y_{\eta}$. 
Since $Y$ is $\tdop$-invariant, it is clear that $W$ is invariant under the $\tdop$-action. 
{The multiplicative torus $T=\tdop_K$ is split over $K$ and hence the vector space $W_K$ has a simultaneous eigenbasis  for the $T$-action. This is well-known in the theory of toric varieties over a field and follows from \cite[Proposition  III.8.2]{Bor}. Note that this $K$-basis is semi-invariant. We will show in the next lemma that such a basis exists as a basis of $W$ over $\kcirc$. } 
\end{art}

\begin{lemma}\label{sec}
$W$ is a free $K^\circ$-module of finite rank which has a semi-invariant basis.
\end{lemma}
\begin{proof}
A valuation ring is a Pr\"ufer domain. Since $W$ is a saturated $\kcirc$-submodule of the free module $V$ of finite rank, we conclude that $W$ is free of finite rank $r$ (see \cite[ch. VI, \S 4, Exercise 16]{Bou}).  {We have seen above that $W_K$ has a simultaneous eigenbasis $w_1, \dots ,w_r$ for the $T$-action. For $j=1, \dots, r$, we have $S_t(w_j)=\chi^{u_j}(t)\cdot w_j$} for all $t \in T(K)$ and some $u_j \in M$. Let $E_{u_j}$ be the corresponding eigenspace. Then $W_{u_j}:=E_{u_j} \cap V$ is a saturated $\kcirc$-submodule of $W$. The same argument as above shows that $W_{u_j}$ is a free $\kcirc$-module of finite rank. We may choose the simultaneous eigenbasis $w_1, \dots ,w_r$ above in such a way that a suitable subset is a $\kcirc$-basis  of $W_{u_j}$ for every $j=1, \dots, r$. Note that every $w_j$ is semi-invariant.

For $t$ in the subgroup $U:=\tdop(K^\circ)=T^\circ(K)$ of $T(K)$, we have $S_t \in GL(V,K^\circ)$ and hence the eigenvalues $\chi^{u_j}(t)$ have valuation $0$. Using reduction modulo the maximal ideal $K^{\circ \circ}$ of $K^\circ$,  the $U$-action becomes a $\tdop_{\widetilde{K}}$-operation on $\widetilde{W}:=W \otimes_{K^\circ} \ktilde$. We note that the reduction of a $K^\circ$-basis in $W_{u_j}$  is linearly independent in $\widetilde{W}$. Using that eigenvectors for distinguished eigenvalues are linearly independent, we conclude that the reduction of $w_1, \dots, w_r$ is a a simultaneous eigenbasis for the  $\tdop_{\widetilde{K}}$-action on $\widetilde{W}$. By Nakayama's Lemma, it follows that $w_1, \dots, w_r$  is a $K^\circ$-basis for $W$. 
\end{proof}

\vspace{2mm}
We can now prove the following quasi-projective version of Sumihiro's theorem.

\begin{prop} \label{quasiprojective Sumihiro}
Let $\Ucal$ be a $\tdop$-invariant open subset of $\Ycal_{A,a}$. Then every point of $\Ucal$ has a $\tdop$-invariant open affine neighbourhood in $\Ucal$.
\end{prop}

\proof Let $z \in \Ucal$ and let $Y:=\Ycal \setminus \Ucal$. Since $Y$ is $\tdop$-invariant, we are in the setting of \ref{setting for quasiprojective Sumihiro} and we will use the notation from there. In particular, we have $s_0 \in H^0(\pdop_\kcirc^{R}, \Ocal(k))$ such that $s_0|_Y=0$ and $s_0(z)\neq 0$. Using  Lemma \ref{sec}, we conclude that there is a  semi-invariant  $s_1 
\in H^0(\pdop_\kcirc^{R}, \Ocal(k))$ with  $s_1(z)\neq 0$ and $\lambda s_1|_Y=0$ for some $\lambda \in \kcirc \setminus \{0\}$.

To construct the affine invariant neighborhood of $z$, we assume first that $z$ is contained in the generic fibre of $\Ucal$ over $\kcirc$. 
Then $\Ucal_1:=\{x\in \Ycal_{A,a} \mid \lambda s_1(x)\neq 0\}$ is  an affine open subset of $\Ucal$ that contains $z$. Since $s_1$ is semi-invariant, it follows that $\Ucal_1$ is $\mb{T}$-invariant proving the claim.  

Now we suppose that $z$ is contained in the special fibre $\Ucal_s$. Let $\zeta$ be a generic point of an irreducible component $Z$ of $Y_s$. Since $Y$ is $\tdop$-invariant, $\zeta$ is the generic point of an orbit whose closure $Z$ does not contain $z$. 
Using the orbit--face correspondence from Proposition \ref{projectivetoric}, there is a projective coordinate $x_{i(\zeta)}$ such that $x_{i(\zeta)}(Z) = 0$ but $x_{i(\zeta)}(z) \neq 0$. By definition of $\Ycal_{A,a}$, we may view $x_{i(\zeta)}$ as a semi-invariant global section of $\Ocal(1)$ on $\pdop_\kcirc^R$. Letting $\zeta$ varying over the generic points of the irreducible components of $Y_s$, we get   a semi-invariant global section $s:= s_1 \cdot \prod_\zeta s_{i(\zeta)}$ of a suitable tensor power of $\Ocal(1)$ on $\pdop_\kcirc^R$ with $s(z) \neq 0$ and $s|_Y=0$. Then $\Ucal_1:=\{x \in \Ycal_{A,a} \mid  s(x)\neq 0\}$ is a $\tdop$-invariant affine open neighbourhood of $z$ in $\Ucal$. \qed

\begin{proof}[Proof of Theorem \ref{sumihiro1}] We are now ready to prove Sumihiro's theorem for a normal $\tdop$-toric variety $\Ycal$ over $\kcirc$. Every point $z \in \Ycal$ has an affine open neighbourhood $\Ucal_0$.  Let $\Ucal$ be the smallest $\tdop$-invariant open subset of $\Ycal$ containing $\Ucal_0$. By Proposition \ref{the A-structure}, there is an equivariant open immersion $i: \Ucal \rightarrow \Ycal_{A,a}$ for suitable $A \in M^{R+1}$ and height function $a$. By Proposition \ref{quasiprojective Sumihiro}, there is a $\tdop$-invariant open neighbourhood $\Ucal_1$ of $z$ in $i(\Ucal)$. We conclude that $i^{-1}(\Ucal_1)$ is an affine $\tdop$-invariant open neighbourhood of $z$ in $\Ucal$ proving Sumihiro's theorem. 
\end{proof}

Finally in order to complete the picture which gives rise to the interplay between toric geometry and convex geometry, we prove Theorem \ref{sumihiro2} which give us a bijective correspondence between normal $\mb{T}$-toric varieties and $\Gamma$-admissible fans.

\begin{proof}[Proof of Theorem \ref{sumihiro2}] 
We assume first that $v$ is not a discrete valuation. 
For simplicity, we fix torus coordinates on the split torus $\tdop$ of rank $n$. 
Let $\ms{Y}$ be a normal $\mb{T}$-toric variety. By Theorem \ref{sumihiro1}, $\Ycal$ has an open covering $\{\Vcal_i\}_{i \in I}$ by affine $\tdop$-varieties $\Vcal_i$. By Theorem \ref{cone}, we have $\Vcal_i \cong \Vcal_{\sigma_i}$ for a $\Gamma$-admissible cone $\sigma_i$ in $\rdop^n \times \rdop_+$ for which the vertices of  $\sigma_i \cap (\rdop^n \times \{1\})$ are contained in $\Gamma^n \times \{1\}$. Since $\Ycal$ is separated, $\Vcal_{ij}:=\Vcal_i \cap \Vcal_j$ is affine for every $i,j \in I$. We conclude that $\Vcal_i \cap \Vcal_j$ is an affine normal $\tdop$-toric variety and hence Theorem \ref{cone} again shows $\Vcal_{ij} \cong \Vcal_{\sigma_{ij}}$ for a $\Gamma$-admissible cone $\sigma_{ij}$ in $\rdop^n \times \rdop_+$. Applying the orbit--face correspondence from \cite[Proposition 8.8]{gubler12} to the open immersions $\Vcal_{ij} \rightarrow \Vcal_{i}$ and $\Vcal_{ij} \rightarrow \Vcal_{j}$, it follows that $\sigma_{ij}$ is a closed face of $\sigma_i$ and $\sigma_j$. Moreover, the same argument shows that $\sigma_{ij} = \sigma_i \cap \sigma_j$ and hence the closed faces of all $\sigma_i$ form a $\Gamma$-admissible fan $\Sigma$ in $\rdop^n \times \rdop_+$ with $\Ycal_\Sigma \cong \Ycal$. From Theorem \ref{cone}, we get now easily the desired bijection. If $v$ is a discrete valuation, then the same argument works if we omit the additional condition on the vertices of the cones.
\end{proof}

\bibliographystyle{plain}
\bibliography{gubler_soto}

{\small Walter Gubler, Fakult\"at f\"ur Mathematik,  Universit\"at Regensburg,
Universit\"atsstrasse 31, D-93040 Regensburg, walter.gubler@mathematik.uni-regensburg.de}

{\small Alejandro Soto,  {Department of Mathematics, KU Leuven, Celestijnenlaan 200B, B-3001
Heverlee, alejandro.soto@wis.kuleuven.be}}

\end{document}